\documentclass[10pt]{amsart}
\usepackage{amsmath,amsfonts,amsthm,amssymb,amscd,stmaryrd,url, longtable}
\usepackage{color}
\usepackage{caption}
\usepackage[top=3.0cm,left=3.0cm,right=3.0cm,bottom=3.0cm]{geometry}
\usepackage[textwidth=30mm]{todonotes}
\newcommand{\s}{\sigma}

\DeclareFontFamily{OT1}{rsfs}{}
\DeclareFontShape{OT1}{rsfs}{n}{it}{<-> rsfs10}{}
\DeclareMathAlphabet{\mathscr}{OT1}{rsfs}{n}{it}

\newtheorem{prop}{Proposition}[section]
\newtheorem{thm}[prop]{Theorem}
\newtheorem{cor}[prop]{Corollary}
\newtheorem{lem}[prop]{Lemma}

\theoremstyle{definition}
\newtheorem*{defn*}{Definition}
\newtheorem{Rem}[prop]{Remark}
\numberwithin{equation}{section}
 
\usepackage{hyperref}
\hypersetup{
  colorlinks   = true, 
  urlcolor     = blue, 
  linkcolor    = blue, 
  citecolor   = purple 
}

\title[Sharper bounds for the Chebyshev function $\psi(x)$]{Sharper bounds for the Chebyshev function $\psi(x)$}

\author{Andrew Fiori, Habiba Kadiri, Joshua Swidinsky}
\address{Department of Mathematics and Statistics, University of Lethbridge, Canada}
\email{andrew.fiori@uleth.ca, habiba.kadiri@uleth.ca, jds26@sfu.ca}
\thanks{Andrew Fiori graciously acknowledges both the University of Lethbridge for their financial support (start-up grant and ULRF) as well as the support of NSERC Discovery Grant RGPIN-2020-05316.
Habiba Kadiri acknowledges both the University of Lethbridge for their financial support (ULRF) as well as the support of NSERC Discovery Grant RGPIN-2020-06731.
Joshua Swidinsky acknowledges receiving financial support through an NSERC USRA and from the University of Lethbridge for this work.
Computations performed in this work were conducted using infrastructure supported by WestGrid (\url{www.westgrid.ca}) and Compute Canada (\url{www.computecanada.ca}).
}


\begin{document}
\maketitle

\begin{abstract}
We improve the unconditional explicit bounds for the error term in the prime counting function $\psi(x)$. In particular, we prove that, for all $x>2$, we have
\[
\left| \psi(x)-x \right|   <  9.22022 \, x \, (\log x)^{3/2} \exp(-0.8476836\sqrt{\log x}),
\]
and that, for all $x \ge \exp(3\,000)$,
\[
\left| \psi(x)-x \right|  <  4.9678 \cdot 10^{-15} x.
\]
This compares to results of Platt and Trudgian (2021) who obtained $4.51\cdot 10^{-13} x $.
Our approach represents a significant refinement of ideas of Pintz which had been applied by Platt and Trudgian. Improvements are obtained by splitting the zeros into additional regions, carefully estimating all of the consequent terms, and a significant use of computational methods. 
Results concerning $\pi(x)$ will appear in the follow up work \cite{FioKadSwi22+}.
\end{abstract}

\section{Introduction}

\subsection{Context}

The main topic of this paper is to look at bounds on the Chebyshev prime counting function 
\[ \psi(x) = \sum_{p^r \le x} \log(p) .\]
We introduce
\begin{equation}\label{def-et-pnt}
E_{\psi}(x)= \Big|\frac{\psi(x)-x}{x}\Big|.
\end{equation}
The prime number theorem for $\psi(x)$, as proven independently by de la Vall\'ee Poussin and Hadamard in 1896 says
\begin{equation}\label{pnt}
\lim_{x\to\infty} E_{\psi}(x) = 0. 
\end{equation}
This was a consequence of the Riemann zeta function not vanishing on the vertical line $\Re s=1$. 
We recall the other prime counting functions and their associated remainder
\begin{equation}\label{def-et-pnt-theta-pi}
\theta(x) = \sum_{p\le x} \log p,\ 
E_{\theta}(x)= \Big|\frac{\theta(x)-x}{x}\Big|,\ 
\pi(x)= \sum_{p\le x} 1,\ 
\ \text{ and }\ 
E_{\pi}(x)= \Big|\frac{\pi(x)-\textrm{Li}(x)}{\frac{x}{\log x}}\Big|,
\end{equation}
where $\textrm{Li}(x)$ is the integral function $\textrm{Li}(x) = \int_2^x \frac{d\,t}{\log t}$.
Equivalent versions of the prime number theorem state that, for $\theta(x)$ and $\pi(x)$ respectively
\begin{equation}\label{pnt-theta-pi}
\lim_{x\to\infty} E_{\theta}(x) = 0, \ \text{and} \ \lim_{x\to\infty} E_{\pi}(x) = 0.
\end{equation}

In 1899, de la Vall\'ee Poussin's zero free region 
\[
\zeta(s) \neq0 \ \text{ for }\ \Re s > 1-\frac{1}{R'\log(|t|+2)}\ \text{ and some }R'>0,
\] 
allowed him to prove that the error term $E_{\psi}(x)$
in the estimate \eqref{pnt}, satisfies
\begin{equation}\label{bnd-et-pnt}
E_{\psi}(x)= \mathcal{O}\big( \exp\big( -c \sqrt{\log x}\big)\big),
\end{equation}
and we have the same estimate for $E_{\theta}(x) $ and $E_{\pi}(x) $.

Explicit bounds for $E_{\psi}(x)$ go back to both separate and joint work of Rosser and Shoenfeld \cite{Rosser,RosserSchoenfeld1,RosserSchoenfeld2,Schoenfeld}.
Recently a number of different theoretical approaches have been applied to prove bounds of various forms:
\begin{enumerate}
\item 
$ E_{\psi}(x) \le a (\log x)^{b} \exp\big( -c \sqrt{\log x}\big)$ for all $x \ge x_0$, where $a,b,c$ are computable (see \cite{Tru16, Bu18, Naz18,PlaTru21ET}).
\item 
$ E_{\psi}(x) \le  \frac{\eta_k}{(\log x)^k}$ for all $x \ge x_0$, where $\eta_k$ is computable in terms of $x_0$ and $k$ (see \cite{Dus18, Naz18,BKLNW21}),
\item 
$ E_{\psi}(x) \le \varepsilon_0$ for all $x \ge x_0$, where $\varepsilon_0$ is computable in terms of $x_0$ (see \cite{FabKad15,Tru16, Dus18, Naz18, Bu18, BKLNW21}).
\end{enumerate}

While all forms can be valid for all sizes of $x$, $(1)$ is naturally arising in the explicit formula (via Perron's formula) relating primes and zeros of the zeta functions. It is also sharper as $x$ gets larger. 
In comparison, methods to obtain $(2)$ have been tuned in to give best bound possible for smaller values of $x$. 
For instance, as it is noted in \cite[Appendix $A_3$]{BKLNW21}, the bound $\varepsilon_0 $ calculated following B\"uthe's methods \cite{Bu16,Bu18,Joh22} is sharper than the one obtained from  \cite{PlaTru21ET} for all $x\ge x_0$ when $x_0\le e^{2\,313}$. 
Note that the last form $(3)$ the most commonly used in applications and that it is often verified as a consequence of either $(1)$ or $(2)$. We refer the reader to \cite{BKLNW21} for a comprehensive and complete description of how to obtain such bounds for $\theta(x)$\footnote{These bounds for $\theta(x)$ are the ones needed for applications, but are actually also valid for $\psi(x)$ as the difference between both is taken care of by the rounding up which has a larger effect than the approximate $\sqrt{x}$ difference between $\psi(x)$ and $\theta(x)$. },
 no matter the size of $x$ and for values of $k$ that are most widely used. 

\subsection{Main Results}

This paper establishes two types of explicit bounds on $\psi(x)$:

\begin{thm}\label{thm-numericVersion}
For any $x_0$ with $\log x_0>1\,000$,  and all $0.9 < \sigma_2 < 1$, $2 \le  c \le 30$,  and $N,\, K\ge 1$ the formula $\varepsilon(x_0)  := \varepsilon(x_0,\sigma_2,c,N,K)$ as defined in \eqref{def-epsilon} gives an effectively computable bound
\[ E_{\psi}(x)  \le  \varepsilon(x_0) \quad \text{ for all }\ x \ge x_0. \]
Moreover, a collection of values, $\varepsilon(x_0)$ computed with well chosen parameters are provided in  Table \ref{TableResultsNumeric}. 
\end{thm}

\begin{thm}\label{thm-asymptoticVersion}
Let $R \leq  5.573412$ be a value for which $\zeta(s)$ has no zeros for $\mathfrak{Re} s \ge 1-\frac{1}{R\log \mathfrak{Im} s}$ for $\mathfrak{Im} s \ge 3$.
For $\log(x_0) > 1\,000$ and with $A(x_0)$ defined as in \eqref{def-Ax} we have 
\begin{equation}
E_{\psi}(x)  \le   \varepsilon_{\text{asm}}(x_0,x) =  A(x_0) \Big( \frac{\log x}{R}\Big)^{3/2} \exp\Big( -2 \sqrt{\frac{\log x}{R}}\Big) \ \text{ for all }\ x \ge x_0.
\end{equation}
In particular with $R= 5.5666305$ we have that each pair $(x_0,A(x_0))$ from Table  \ref{TableResultsAsympt} gives such a bound.
\end{thm}
\begin{Rem}
The formula for $A(x_0)$ from \eqref{def-Ax} generally increases as $R$ decreases. 
Note that since this paper was submitted, Mossighoff, Trudgian, and Yang \cite{MoTrYa22} have announced a reduced value of $5.558691$ for $R$. The current tables do not reflect this improvement. 
\end{Rem}

\noindent The proofs of Theorems \ref{thm-numericVersion} and \ref{thm-asymptoticVersion} are given in Sections \ref{sec:numeric} and \ref{sec:asymptotic} respectively.

\begin{cor}\label{cor:allx}
For all $x > 2$ we have
\[
E_{\psi}(x)   <  9.22022(\log x)^{3/2}\exp(-0.8476836\sqrt{\log x}). \]
\end{cor}
\noindent The proof of Corollary \ref{cor:allx} is given in Section \ref{sec:asymptotic}.

This result provides a significant improvement to the state of the art with regard to a bound valid for essentially all $x$.
 
 \begin{Rem}[Comparison with previous similar results]\label{compprev}\ \noindent
\begin{enumerate}
\item
As in \cite{PlaTru21ET}, we use the partial numerical verification of the Riemann Hypothesis $H_0 = 3\cdot 10^{12}$ from \cite{PlaTru21RH}.
As in \cite{CulJoh21} we use $R= 5.5666305$ for the zero-free region following \cite[Theorem 1 and Section 6.1]{TrudMos15}. 
That is if $\varrho=\beta+i\gamma$ is a non trivial zero of the Riemann zeta function, then 
\begin{equation}\label{verifRH}
\beta =\frac12 \ \text{ if }\ |\gamma|\le H_0
\end{equation}
and 
\begin{equation}\label{zfr}
\beta \le 1- \frac1{R\log |\gamma|}\ \text{ if }\ |\gamma| >H_0.
\end{equation}
\item \label{subconv}
It has recently been discovered that the Hiary \cite{Hia15} subconvexity bound for 
\[|\zeta(1/2+it)|\le a_1 t^{1/6} \log t,\] with $a_1=0.63$ relies on an erroneous explicit version of van der Corput second derivative test due to Cheng-Graham \cite{ChGr04}.  
See \cite[Footnote 3]{Pat21+} for details. 
After accounting for this correction, the constant, $a_1$, changes to $0.77$.
Very recently a preprint of Hiary, Patel and Yang \cite{HPY22}  announced that they have recovered and improved this to obtain $a_1= 0.618$.

Because the value $0.63$ was used in  \cite[Theorem 1.1]{KLN18}, the error would render many results concerning the prime number theorem unreliable (for instance \cite{Tru14b, PlaTru15,PlaTru21ET,CulJoh21}).

In the present work we shall use the the worse sub-convexity bound for the Riemann zeta function, namely $0.77$, and still recover, and actually improve on these previous results for $E_{\psi}(x)$. 

Additionally, in Table \ref{table:NSigmaOptTable} we re-calculate the bounds from \cite[Theorem 1.1]{KLN18}, again using $0.77$ taking this correction into account.
Given the existence of the preprint \cite{HPY22} we have also included in Table \ref{table:NSigmaOptTable} the results one would obtain using $0.618$. 
Finally, we note that if and when the $0.618$ value is confirmed the results of Corollary \ref{cor:allx} will improve to
\[ 
E_{\psi}(x)   <  8.99284(\log x)^{3/2}\exp(-0.8476836\sqrt{\log x})\ \text{ for all }\ x>2. \]
We note that recent work \cite{PaYa23} by Patel and Yang offers a significant asymptotic improvement to the sub-convexity bound: $|\zeta(1/2+it)|\le 66.7t^{\frac{27}{164}}$ for all $t\ge 3$. This could further improve zero density results using the methods of \cite{KLN18}. Implementing such improvements is a project in progress.
\item
Theorem \ref{thm-numericVersion} compares directly to \cite[Theorem 1, Table 1]{PlaTru21ET},  \cite[Theorem 1.5, Table 1]{CulJoh21}, and \cite[Theorem 1.1, Table 1]{JohYan22}.
For instance, for $x_0=e^{3\,000}$, the admissible value for $\varepsilon(x_0)$ goes from $4.51 \cdot 10^{-13}$ (from \cite{PlaTru21ET}), $1.42 \cdot 10^{-13}$ (from \cite{CulJoh21}), and $3.06 \cdot 10^{-14}$ (from \cite{JohYan22}), to $ 4.9678  \cdot 10^{-15}$.
\item
Theorem \ref{thm-asymptoticVersion} also compares to 
\cite[Theorem 1]{PlaTru21ET}. The later states 
\begin{equation*}
E_{\psi}(x)  \le  A \Big( \frac{\log x}{R}\Big)^{B} \exp\Big( -C \sqrt{\frac{\log x}{R}}\Big) \ \text{ for all }\ x \ge x_0.
\end{equation*}
where values for $(A,B,C)$ are given in \cite[Table 1]{PlaTru21ET}.
We prove that the expected values $B=3/2$ and $C=2$ are actually admissible values for all $x_0$ (replacing larger bounds such as $B=1.51, C=1.94$ when $x_0=e^{10\,000}$),
In addition we also reduce the values for $A$, finding for instance directly $A=176.3$ (and indirectly, see Corollary \ref{cor:allx}, $A=121.2$) instead of $535.4$ when $x_0=e^{10\,000}$.
Very recently \cite[Table 1]{CulJoh21} provided improvements on the constants $A$, and slight improvements on $B$ and $C$.
For example at $x_0=e^{10\,000}$ they obtained $A=268.6$,  $B=1.508$, and $C=1.957$.
\item
Finally, Corollary \ref{cor:allx}, which gives the bound 
\[
E_{\psi}(x)   < 9.2203 (\log x)^{1.5}\exp(-0.84768\sqrt{\log x}) \ \text{ for all }\ x>2,
\] can be compared against bounds obtainable from \cite{PlaTru21ET} and \cite{CulJoh21}.
The result is both asymptotically better, point wise better, and valid on a larger region, than the previous bounds. 
For example from \cite{PlaTru21ET} one can conclude the bounds of
\[  27.8755(\log x)^{1.52} \exp(-0.80057\sqrt{\log x}),\ \text{for all}\   x \ge \exp(3\,000).\]
Whereas the recent work of \cite{CulJoh21} would give 
\[ 19.6454  (\log x)^{1.514} \exp(-0.815895\sqrt{\log x}),\ \text{for all}\  x \ge \exp(3\,000).\]
More recently, Johnston and Yang \cite{JohYan22} have proven
\begin{align*}
8.86  (\log x)^{1.514} \exp(- 0.8288\sqrt{\log x}),&\ \text{for all}\  x \ge \exp(3\,000), \\
6.72  (\log x)^{1.508} \exp(- 0.8369\sqrt{\log x}),&\ \text{for all}\  x \ge \exp(10\, 000).
\end{align*}
For larger values of $x$, they use Korobov-Vinogradov's zero-free region (as made explicit by Ford \cite{For01}), so their results improve. For instance, they obtain:
\[
23.13 (\log x)^{1.503} \exp(- 0.8659\sqrt{\log x}), \ \text{for all}\  x \ge \exp(100\, 000).
\]
We also note the work \cite{Ya23} of Yang who provided an explicit Littlewood type zero-free region. This would potentially lead to improvements in the intermediate region between where the Korobov-Vinogradov's zero-free region and the de la Vall\'ee Poussin region are the best.
 \end{enumerate}
 \end{Rem}
 
 A discussion of the sources of the various improvements appears in Section \ref{improvements}.

\subsection{Acknowledgements}
The authors thank Nathan Ng for his important feedback in the preparation of this manuscript.
We also would like to thank him for providing a correction and improvement to Dudek's explicit version of Riemann-Von Mangoldt's formula \cite[Thm 1.3.]{Dud16}. In the end, we do not use this  unpublished result, but instead use Cully-Hugill and Johnston's \cite[Thm 1.2.]{CulJoh21}. We note however that either result lead to the same results presented in this article. We are grateful to Cully-Hugill and Johnston for updating us with their progress and sharing a preprint of their work with us.
Finally, we are thanking the referee for their careful reading of our manuscript.

\section{Preliminary results about the zeros of the Riemann-Zeta function}\label{section-zeros}

We first fix some notation. 
We recall that $\zeta(s)$ denotes the Riemann-Zeta function as given by the series
\[
\zeta(s) = \sum_{n\ge 1} n^{-s} \ \text{ for }\ \Re s>1.
\]
It can be analytically continued to the whole complex plane: it has a simple pole of residue 1 at $s=1$, and zeros at negative even integers, the remaining non-trivial zeros being located in the vertical strip $0<\Re s<1$. 
Throughout this paper, we denote by $\rho = \beta + i\gamma$ such a non-trivial zero of $\zeta(s)$ with real part $\beta$ and imaginary part $\gamma$. 
The Riemann Hypothesis is the conjecture that all such zeros lie on the vertical line $\Re s=1/2$.
Since Riemann, calculations for the first zeros have confirmed this hypothesis. We denote $H_0$ a verification height of the Riemann Hypothesis; that is, if $\rho$ is a zero of $\zeta(s)$, and
$|\gamma| < H_0$, then $\beta = 1/2$. 
For example, we may take $H_0 = 3 \cdot 10^{12}$ by the recent work of Platt and Trudgian  \cite{PlaTru21RH}.

Let $T\ge 1$ and $\sigma\in (0,1)$. 
We recall that $N(T)$
counts the number of zeros of $\zeta(s)$ in the region $0 < \gamma < T$ and $0 < \beta < 1$, and $N(\sigma, T)$
the number of zeros of $\zeta(s)$ in the region $0 < \gamma < T$ and $\sigma \le \beta < 1$.

\subsection{Estimates of partial sums over the zeros using the count of the number of zeros $N(T)$}\label{section-N(T)}

\begin{lem}\label{b1Bnd}
Let $1 < U < V$ and assume that there exist positive constants $b_1,b_2,b_3$ such that, for every $T \geq 2$,
\[ \Big|N(T) - \Big(\frac{T}{2\pi}\log \frac{T}{2\pi e} + \frac78\Big) \Big| \le R(T),\]
where
\begin{equation*}
  \label{Rdefn}
R(T) := b_1 \log T + b_2 \log \log T + b_3.
\end{equation*}
Then,
\[
\sum_{ U \le \gamma <V  } \frac{1}{\gamma} 
\le B_1(U,V) 
\]
with
\begin{equation}
 \label{def-B1}
 B_1(U,V) = 
\Big(\frac{1}{2 \pi} + \frac{b_1 \log U + b_2}{U \log U \log(U/2 \pi)} \Big) \big( \log (V/U)\, \log (\sqrt{V U}/(2\pi)) \big)  + \frac{2R(U)}{U}.
\end{equation}
\end{lem}
This easily follows from the calculations following \cite[Corollary 2.6]{FabKad15}, which is based on \cite[Lemma 7]{RosserSchoenfeld2}. 
See also \cite[Lemma 3]{BrPlTr21} for another more refined estimate. 
\begin{Rem}
By the classical work of Rosser \cite[Theorem 19]{Rosser} in Lemma \ref{b1Bnd} we may take 
$b_1=0.137$, $b_2=0.443$, and $b_3=1.588$.
We note that these constants can be improved, see for instance the work of  Hasanalizade, Shen, and Wong in \cite{HSW21} who obtained:
\begin{equation}\label{def-ai}
b_1=0.1038,\ 
b_2=0.2573,\ \text{ and }\ 
b_3=9.3675.
\end{equation}
In any case, in the present application the difference to the final result is negligible.
\end{Rem}

Assuming that we have a list of first zeros on the line, namely verifying $0<\gamma<T_0$ for some fixed value $T_0$, then we can refine the estimate from Lemma \ref{b1Bnd}:
\begin{cor}\label{prop:SOTO}
For each pair $T_0,\,S_0$ as in Table \ref{tab:Sig1} we have that, for all $V> T_0$,
\begin{equation}
\sum_{ 0 < \gamma < V  } \frac{1}{\gamma} < S_0 + B_1(T_0,V),
\end{equation}
where $B_1$ is defined \eqref{def-B1}.
\end{cor}
\begin{proof}
Each row in Table \ref{tab:Sig1} represents a strict upper bound for the numerical evaluation of 
\begin{equation} \label{def:S0}   S_0 - 10^{-10} < \sum_{ 0< \gamma < T_0  } \frac{1}{\gamma}  < S_0 .\end{equation}
We use the zeros as computed by Platt  \cite{Pla17} and made available
through LMFDB \cite{LMFDB}. 
We then apply Lemma \ref{b1Bnd} to bound the remain sum $\sum_{ T_0 \le \gamma < V  } \frac{1}{\gamma}$.
\end{proof}

\begin{table}[h]
\caption{Tight upper bounds for values of reciprocal sums of zeros up to height $T_0$ using Platt's tabulations, see \eqref{def:S0}.}\label{tab:Sig1}
\begin{tabular}{|r|r|}
\hline
$T_0$\ \qquad\ 	&$S_0$\ \qquad\ \ \qquad\ \\
\hline
100&	0.5922435112\\
1\,000&	2.0286569752\\
10\,000	&4.3080354951\\
100\,000&	7.4318184970\\
1\,000\,000&	11.3993199147\\
10\,000\,000&	16.2106480369\\
100\,000\,000&	21.8657999924\\
1\,000\,000\,000&	28.3647752011\\
10\,000\,000\,000&	35.7075737123\\
30\,610\,046\,000&	39.5797647802\\
\hline\end{tabular}
\end{table}

\begin{Rem}
Note that in the current application, $S_0 + B_1(T_0,V)$ is a better approximation of the sum than $B_1(0,V)$. However, the difference between these estimates does not significantly impact the result as they become insignificant part of the error once the factor $x^{-1/2}$ is introduced.
However it could be valuable if one were to consider much smaller values $x$.
\end{Rem}

\subsection{Estimates of partial sums over zeros close to the edge $\Re s=1$ using the count of the number of zeros $N(\s,T)$}
\label{section-N(s,T)}

\begin{defn*}[Zero-density bound]\label{ZDB}
We say that $N(\s,T)$ satisfies (ZDB) if there exists $ \s_0>0$ such that
\begin{equation}\label{bnd-ZDB}
 N(\s,T) \le  \tilde{N}(\s,T) \ \text{ for all }\  \s >\s_0 ,
 \end{equation}
where $ \tilde{N}(\s,T) $ is of the form
\begin{equation}\label{def-Ntilde}
 \tilde{N}(\s,T) = c_1 T^{p} (\log T )^{q} + c_2(\log T)^2,
  \end{equation}
for some positive functions $c_1=c_1(\s), c_2=c_2(\s), p=p(\s)$, and $q=q(\s)$ with $0<p(\s)<1$.
\end{defn*}

We also require a preliminary result estimating sums over the zeros close to the $1$-line, of the form
\begin{equation}
\label{def-s0}
s_0(\s, U,V) = \sum_{\substack{U \le \gamma < V \\ \s \le \beta < 1}} \frac{1}{\gamma} .
\end{equation}
This result is a particular case of the studies done in \cite{KLN20+}, where Kadiri, Lumley, and Ng provide estimates for sums of the form $\sum_{\substack{U \le \gamma < V \\ \s \le \beta < 1}} \frac{\nu(\gamma)}{\gamma^{m+1}}$, where $\nu$ is a positive, bounded, continuously differentiable function. 

We recall that the incomplete Gamma function is given by 
\begin{equation}\label{def-Gamma}
\Gamma(s,x) = \int_x^{\infty} t^{s-1} e^{-t} dt, \ \text{ for }\ \Re s>0.
\end{equation}
\begin{lem} 
\label{B4Bound}
Let $T_0\ge 2$ and $\s>0$.
Assume there exists $c_1(\s),c_2(\s),p(\sigma)$, and $q(\sigma)$ for which the zero-density bound \hyperref[ZDB]{(ZDB)} is satisfied for $N(\s,T) $.
Assume $\s \geq 5/8$ and $T_0 \le  U < V $. Then,
\begin{equation}\label{def-RS}
s_0(\s, U,V) 
\le  B_0(\s, U, V),
\end{equation}
where $B_0(\s, U, V)$ is given by 
\begin{equation}\label{def-B4}
\begin{split}
B_0(\s, U, V) = 
& c_1  \frac{ (\log V )^{q }}{V^{1-p }} + c_2  \frac{(\log V)^2}{V}
\\&+ \frac{c_1}{(1-p)^{q+1} } \left( \Gamma(q+1,(1-p)\log U)  - \Gamma(q+1,(1-p)\log V)\right)
\\&+ c_2 \left(  \Gamma(3, \log U)- \Gamma(3, \log V)\right),
\end{split}
\end{equation}
where in this definition, we write $c_1,c_2, p$, and $q$ respectively in place of $c_1(\sigma),c_2(\sigma),p(\sigma)$, and $q(\sigma)$.
\end{lem}
\begin{proof}
We rewrite $s_0$ using the Stieljtes integral definition, then integrate by parts, and use the bound $N\le \tilde N$:
\begin{align*}
s_0(\s, U,V) &
= \int_{U}^{V} \frac{ dN(\sigma,y) }y 
 = \frac{N(\s,V)}V -  \frac{N(\s,U)}U + \int_{U}^{V}\frac{N(\s,y)}{y^2} dy
 \le  \frac{\tilde N(\s,V)}V + \int_{U}^{V}\frac{\tilde N(\s,y)}{y^2} dy 
\\& =  c_1  \frac{ (\log V )^{q }}{V^{1-p }} + c_2  \frac{(\log V)^2}{V}
+ c_1  \int_{U}^{V}  \frac{ (\log y )^{q } }{y^{2-p}} dy
+ c_2  \int_{U}^{V} \frac{(\log y)^2 }{ y^{2}} dy
\\& =  c_1  \frac{ (\log V )^{q }}{V^{1-p }} + c_2  \frac{(\log V)^2}{V}
+ c_1 \left(J_{q,2-p} (U)-J_{q,2-p} (V)\right)
+ c_2 \left( J_{2,2} (U)-J_{2,2} (V)\right),
\end{align*}
where we define the integral $J_{a,b}$ for $a>0,b>1$ by
\begin{equation}\label{def-Jab}
 J_{a,b}(T)=\int_{T}^{\infty} \frac{\log^a y}{y^b} dy.
\end{equation}
We conclude by recognizing 
\[
 J_{a,b}(T) 
  =  \int_{T}^{\infty} (\log y)^a  y^{-b} dy
  =  \int_{ (\log T)(b-1)}^{\infty}  t^a (b-1)^{-a}  e^{-t} \frac{dt}{b-1} 
=\frac{1}{(b-1)^{a+1} } \Gamma(a+1,(b-1)\log T),
\]
doing the variable change $y=e^{\frac{t}{(b-1)}}$
(so
$(\log y)^a = \frac{t^a}{(b-1)^a}, y^{-b} dy 
= \frac{e^{-t}}{(b-1)}dt$).
\end{proof}
\begin{Rem}
We recall some well-known facts about the incomplete Gamma function:
\begin{equation}\label{prop-Gamma}
\Gamma(3,x) = \Big(x^2+2(x+1)\Big)e^{-x} , \ \text{and}\ 
\Gamma(s,x) \sim x^{s-1} e^{-x} \ \text{as}\ x\to\infty, \ \text{for all }\ s>1.
\end{equation}
We recall that $1-p>0$. 
Assuming $U,V$ are large enough, an estimate for $B_0$ is
\[
\frac{c_1   }{ 1-p   } \Big( 
  \frac{(\log U)^{q }}{U^{ 1-p }} 
-   p     \frac{(\log V)^{q }}{V^{ 1-p }}
\Big) .
\]
\end{Rem}
To apply the above we use the functions $\tilde{N} $ 
obtained in Kadiri-Lumley-Ng's \cite[Theorem 1.1]{KLN18}. 
\begin{thm}  \label{ZD-KLN}
Let $H_0$ denote a verification height for the Riemann Hypothesis.
Let  $\frac{10^9}{H_0} \le k \le 1,d>0, H\in [1002,H_0)$, $\alpha>0$, $\delta\ge1$, $\eta_0=0.23622$, $1+\eta_0\le \mu \le 1+\eta$, 
and $\eta\in(\eta_0,\frac{1}{2})$ be fixed. Let $\sigma > \frac12 +\frac{d}{\log H_0}$. 
Then there exist $\mathcal{C}_1 ,\mathcal{C}_2 >0$ such that, for any $T\ge H_0$,
\begin{align} 
N(\s,T)
& \le \frac{(T-H)(\log T)}{2\pi d}
\log
\Big( 1 
+\frac{  \mathcal{C}_1
(\log (kT))^{2\s}  (\log T)^{4(1-\s)} T^{\frac83(1-\s)}
 }{T-H}  \Big)
+ \frac{ \mathcal{C}_2  }{2\pi d} (\log T)^2 ,
\\
\label{KLN-NsigT1}
N(\s,T)
& \le 
 \frac{ \mathcal{C}_1}{2\pi d}
(\log (kT))^{2\s}  (\log T)^{5-4\s } T^{\frac83(1-\s)}
+ \frac{ \mathcal{C}_2  }{2\pi d} (\log T)^2.
\end{align}
where $\mathcal{C}_1= \mathcal{C}_1(\alpha, d,\delta, k, H, \s)$ and $\mathcal{C}_2=\mathcal{C}_2 (d, \eta, k, H, \mu,\s)$ are defined in  \cite[Lemma 4.14]{KLN18} as
\begin{align}
\label{eqn-C1}
 \mathcal{C}_1(\alpha, d,\delta, k, H, \s) &=  b_{12}(H)e^{\frac{8}{3}\delta(2\s-1)M(k,\delta)+\frac{4\delta(2\s-1)\log\log H_0}{\log(k H_0) + 2\delta }  + \frac{2d(2\log\log H_0 - \log\log (kH_0))}{\log H_0} + \frac{8d}{3} + 2\alpha } \times
                                      \\&\nonumber  \ \qquad\  \mathcal{U}(\alpha,k,H_0)^{2(1-\sigma) + \frac{2d}{\log H_0} + \frac{2\delta(2\sigma-1)}{\log( kH_0) + 2\delta}} 
                                        \mathcal{V}(\alpha,k,\delta,H_0)^{2\sigma-1} ,
\\
\label{eqn-C2} 
 \mathcal{C}_2 (d, \eta, k, H, \mu,\s) & = C_7(\eta,H)\Big(\mu - \sigma + \frac{d}{\log H_0}\Big) + \frac{C_8(k,\mu)}2 .
 \end{align}
The precise definitions of the individual terms $b_{12}(H)$, $M(k,\delta)$, $\mathcal{U}(\alpha,k,H_0)$, $\mathcal{V}(\alpha,k,\delta,H_0)$, $C_7(\eta,H)$, and $C_8(k,\mu)$ can be be found respectively in  \cite[(4.18), (4.19), (4.23), (4.56), (4.63), and (4.75)]{KLN18}.
\end{thm}

\begin{proof}
This is \cite[Theorem 1.1]{KLN18}  except that we may replace $C_8(k,\mu)$ by $C_8(k,\mu)/2$ in \eqref{eqn-C2} as pointed out by \cite[Section 6]{CulJoh21} using \cite[Theorem 2]{CT21}.
Additionally we note that one should use an admissible value of $a_1$ in the subconvexity bound for $\zeta(1/2+it)$ (see Remark \ref{compprev}(\ref{subconv})) and one should correspondingly adjust the value of $a_2$ to be $a_2 = \tfrac{6}{e}a_1-\zeta(1/2)$. 
\end{proof}

We will denote the bound in equation \eqref{KLN-NsigT1} by $\tilde{N}(\sigma, T)$. 

\begin{Rem}
The original table of results published in \cite{KLN18} used $H_0 = 30\,601\,646\,000$ due to Platt \cite{Pla17}. In Table \ref{table:NSigmaOptTable}
we provided updated parameters optimized for $H_0 = 3\cdot 10^{12}$ based on the more recent work of Platt-Trudgian \cite{PlaTru21RH} and the other updates highlighted above.
\end{Rem}

An interesting feature of the formulation of the above theorem is that the optimized results apply to each sigma separately. 
It is useful in many applications to be able to treat $c_1(\s)$ and $c_2(\s)$  as constant on an interval.
The following Corollary allows us to do this:
\begin{cor}\label{ZDinterval}
For each $\s_1,\, \s_2,\, \tilde{c_1},$ and $\tilde{c_2}$ given in Table \ref{table:NSigmaOptTable2}, we have that
$N(\s,T)\le \tilde N(\s,T)$ for all $\s_1 \le  \s \le  \s_2$,
where
\begin{equation}\label{ZD} 
\tilde N(\s,T) =  \tilde{c_1} T^{p(\sigma)} (\log T )^{q(\sigma)}+ \tilde{c_2} (\log T)^2,
\end{equation}
 with $p(\s) = \frac{8}{3}(1-\s)$ and $q(\s) = 5-2\s$.
Consequently, one can use Table \ref{table:NSigmaOptTable2} to define $c_1(\s)$ and $c_2(\s)$ as piece-wise constant functions which satisfy \hyperref[ZDB]{(ZDB)}.

\end{cor}
\begin{proof}
Based on the shape of the formulas \eqref{eqn-C1} and \eqref{eqn-C2}, 
and for any fixed values $\alpha , d, \delta, \eta, k,H, \mu$ satisfying the conditions of the theorem, 
we have 
\[   \mathcal{C}_1(\sigma) =  \mathcal{C}_1(\alpha, d,\delta, k, H, \s) =   d_{11} d_{12}^\sigma 
\ \text{and}\ 
\mathcal{C}_2(\sigma) = \mathcal{C}_2 (d, \eta, k, H, \mu,\s) = d_{21} - d_{22} \sigma, \]
for values $ d_{ij}=d_{ij}(\alpha , d, \delta, \eta, k,H, \mu)$, which depend on the given parameters $\alpha , d, \delta, \eta, k,H, \mu$ but are independent of $\s$. For instance it is easy to identify $d_{22}=C_7(\eta, H)$. 
As functions of $\sigma$, it is clear that the extreme values of $ \mathcal{C}_1(\sigma)$ and $ \mathcal{C}_2(\sigma)$ occur at the endpoints of the interval $[\sigma_1,\sigma_2]$. 
By explicitly computing the values of the $d_{ij}$'s we may identify which of these endpoints gives each extreme value.
Then for the given parameters are given in Table \ref{table:NSigmaOptTable2}, 
as are the results of taking the maximum value of both $ \mathcal{C}_1(\sigma)$ and $ \mathcal{C}_2(\sigma)$ at these endpoints.
\end{proof}

\section{Setting up the argument}
\subsection{Explicit Perron's Formula}
The starting point for our results is an explicit version of Riemann-Von Mangoldt's. We cite here two versions, the first due to Dudek, the second to Cully-Hugill and Johnston. 
\begin{thm}\cite[Theorem 1.3]{Dud16}\label{Dud}
Let $x>e^{50}$ be half an odd integer and suppose that $50<T <x$. Then
\begin{equation} 
\Big|\frac{\psi(x) - x}x\Big| \le \sum_{|\gamma| < T} \Big| \frac{x^{\rho-1}}{\rho} \Big|
+  O^{\star} \left( 2\frac{(\log x)^2}{T} \right),
\end{equation} 
where $f=g+O^*(h)$ means $|f-g|\le h$. 
\end{thm}
\begin{thm}\cite[Theorem 1.2]{CulJoh21}\label{CHJ}
For any $\alpha\in (0,1/2]$ and $\omega\in[0,1]$ there exist constants $M$ and $x_M$ such that for $\max\{51,\log x\} < T < (x^{\alpha}-2)/5$ and some $T^{\star} \in [T, 2.45 T]$,
\begin{equation} 
\psi(x) = x - \sum_{|\gamma|\le T^{\star}} \frac{x^{\varrho}}{\varrho} + O^{\star} \left( M\frac{x}{T}(\log x)^{1-\omega} \right)\ \text{for all}\  x\ge x_M.
\end{equation} 
An admissible value of $\log(x_M), \alpha, \omega$, and $M$ is $(40, 1/2, 0, 2.091)$. More are given in \cite[Table 5]{CulJoh21}.
\end{thm}
\begin{Rem}\label{rem:perron}
We note that the above work \cite{CulJoh21} of Cully-Hugill and Johnston is currently being peer reviewed, whereas the earlier articles \cite{Dud16} and \cite{Ram16} have errors\footnote{For example \cite[Lemma 2.3]{Dud16} is wrong as stated, and used, in particular $\frac{\zeta'(s)}{\zeta(s)}$ is unbounded at negative even integers where the lemma is applied. Additionally, \cite[Section 2.1]{Dud16} references in several places Lemma 2.7, when it needs to use Lemma 2.8 so as to bound $I_7$. The obvious correction introduces a factor of $20$ onto one of the terms. In \cite{CulJoh21}, Cully-Hugill and Johnston point out errors in Ramar\'e's \cite{Ram16}.}. 
While these errors can be corrected, doing so may not recover the claimed results.
Dudek's result was used in Platt and Trudgian's \cite{PlaTru21ET} and represented a non-negligible part of the error term for $\psi(x)$ there.
However, in our present application we shall eventually take $T$ large enough that this term becomes negligible. 
As the choice of version we use would ultimately not impact our results, we have decided to use the worse bound published while including the more stringent hypothesis of \cite[Theorem 1.2]{CulJoh21}.
The bound of Proposition \ref{Perron-exp} is Dudek's bound as stated in \cite[Thm 1.3]{Dud16} and is worse by a factor of $(\log x)$ than Cully-Hugill and Johnston's. 

It is worth noting that all of the asymptotic results and numerical  results are still obtainable if we instead multiply $\varepsilon_1$ by a factor of a $100$, or even a factor of $ \log x $.
 In most cases this requires no additional work. 
For a comparison between these versions of Perron's formula for small $x$, see Table \ref{tab:smallx}.
\end{Rem}
\begin{prop}\label{Perron-exp}
Let $x> e^{50}$ and $3\log x<T< \frac{\sqrt x}{3}$. Then 
\begin{equation}\label{bnd-Epsi}
\Big|\frac{\psi(x) - x}x\Big| \le \sum_{|\gamma| < T} \Big| \frac{x^{\rho-1}}{\rho} \Big|
+  \varepsilon_1(x,T) ,
\end{equation}
with 
\begin{equation}\label{def-eps1}
\varepsilon_1(x,T) = 2\frac{(\log x)^2}{T}.
\end{equation}
\end{prop}

\subsection{Splitting the sum over the zeros }\label{section-split}
Following Proposition \ref{Perron-exp}, it remains to estimate the sum over the zeros.
Fixing $\sigma_1$ and $\sigma_2$ with $1/2 < \sigma_1 < \sigma_2 < 1$, 
we define
\begin{equation}\label{Sigmaab}
\Sigma_{a}^{b}= 2\sum_{\substack{0<\gamma<T \\ a \le \beta <  b}} \frac{x^{\beta-1}}{\gamma},
\end{equation}
and write 
\begin{equation}\label{bndSums}
 2 \sum_{0<\gamma < T} \frac{x^{\beta-1}}{\gamma}=
\Sigma_{0}^{\sigma_1}+\Sigma_{\sigma_1}^{\sigma_2}+\Sigma_{\sigma_2}^1.
\end{equation}
We improve upon this result by splitting our sum in a further way.
\begin{Rem}\ \noindent
\begin{enumerate}
\item 
In the work of Platt-Trudgian, Pintz's method is used to break the sum with $\sigma_1=\sigma_2$. 
Our approach of splitting once further will bring significant savings. This is because it allows us to reduce the contribution of the large number of zeros on (or relatively near) the half line by keeping them less than $\sigma_1$
while at the same time allowing us to move $\sigma_2$ towards the $1$-line. Moving $\sigma_2$ towards the $1$-line is necessary to effectively deal with the interaction of the shape of the zero free region with the zero density bounds.
Because of how we tightly (relative to the quality of the zero density) bound the contribution between $\sigma_1$ and $\sigma_2$ further splitting of this region is not beneficial.

\item 
We will ultimately take $\sigma_1=0.9$, this choice is arbitrary and any value between $5/8$ and $0.95$ would yield similar numerical  and asymptotic results.
For smaller values of $x$ (for instance those less than $e^{1\,000}$) taking $\sigma_1$ smaller would likely be slightly beneficial.
\item 
We will need to optimize the choice of $\sigma_2$, we shall see in Section \ref{sec:asymptotic} that asymptotically one wants to take $\sigma_2$ slightly larger than $1-\frac{2}{R\sqrt{\log x}}$ to balance the contribution from 
$\Sigma_{\sigma_1}^{\sigma_2}$ and $\Sigma_{\sigma_2}^1$. 
\end{enumerate}
\end{Rem}

In the following sections, we describe methods for bounding each $ \Sigma_{0}^{\sigma_1}, \Sigma_{\sigma_1}^{\sigma_2}, \Sigma_{\sigma_2}^1$ sum individually.
To do this will require different facts about the zeros of the zeta function.

\subsection{Bounds for $\Sigma_{0}^{\sigma_1}$}

\begin{prop}\label{bnd-Sigma-0-sigma1}
Let $\sigma_1 \in (\frac{1}{2}, 1)$ and let $(T_0,S_0)$ be taken from Table \ref{tab:Sig1}.
 Then,
\begin{equation}
\Sigma_0^{\sigma_1} = 2\sum_{\substack{0<\gamma<T \\ 0 \le \beta <  \sigma_1}} \frac{x^{\beta-1}}{\gamma}
\le  \varepsilon_2(x, \sigma_1, T_0, T),
\end{equation}
where 
\begin{equation}\label{def-epsilon2}
\varepsilon_2(x, \sigma_1, T_0, T) =  2 x^{-1/2} (S_0 + B_1(T_0,T))
+ \Big( x^{\sigma_1-1} - x^{-1/2} \Big)B_1(H_0, T),
\end{equation}
and $B_1$ defined in \eqref{def-B1}.
\end{prop}

\begin{proof}
Break $\Sigma_0^{\sigma_1}$ into two parts, first, the area with $1/2 < \beta < \sigma_1$ and second the part with $\beta \le  1/2$.
In the first we use that $x^{\beta-1} \le  x^{\sigma_1-1}$, and in the second that $x^{\beta-1} \le  x^{-1/2}$.
For the zeros satisfying $\beta > 1/2$, the partial verification of the Riemann Hypothesis implies $\gamma\ge H_0$. 
Thus
\begin{equation}\label{s2SplitSum}
\Sigma_0^{\sigma_1}
= 2\sum_{\substack{ 0 < \gamma < T \\ \beta < \sigma_1}} \frac{x^{\beta-1}}{\gamma} 
= 2\sum_{\substack{ 0 < \gamma < T \\ \beta \le  1/2}} \frac{x^{\beta-1}}{\gamma}
+ 2\sum_{\substack{ 0 < \gamma < T \\ 1/2 < \beta < \sigma_1}} \frac{x^{\beta-1}}{\gamma}
\le2 x^{-1/2}\sum_{\substack{ 0 < \gamma < T \\ \beta \le  1/2}} \frac{1}{\gamma}
+ 2 x^{\sigma_1-1} \sum_{\substack{ H_0 \le \gamma < T \\ 1/2 < \beta < \sigma_1}} \frac{1}{\gamma}.
\end{equation}
For the first sum, we recognize that
\[
\sum_{\substack{ 0 < \gamma < T \\ \beta \le  1/2}} \frac{1}{\gamma}
=
\sum_{\substack{ 0 < \gamma < T }} \frac{1}{\gamma}
- 
\sum_{\substack{ H_0 \le \gamma < T \\ \beta >  1/2}} \frac{1}{\gamma}
\le 
\sum_{\substack{ 0 < \gamma < T }} \frac{1}{\gamma}
- 
\sum_{\substack{ H_0 \le \gamma < T \\ 1/2< \beta < \sigma }} \frac{1}{\gamma}.
\]
For the second sum, we use the symmetry of zeros accross $\Re s=1/2$: for every zero $\beta + i\gamma$ that is being counted on the left sum, we also count the symmetric zero $(1-\beta) + i\gamma$ on the right sum:
\[
\sum_{\substack{ H_0 \le \gamma < T \\ 1/2 < \beta < \sigma_1}} \frac{1}{\gamma}
\le 
\frac12\sum_{\substack{ H_0 \le \gamma < T  }} \frac{1}{\gamma}.
\]
 Thus \eqref{s2SplitSum} becomes
\begin{equation}\label{s2SplitSum1}
\Sigma_0^{\sigma_1}
\le 2x^{-1/2} \sum_{\substack{ 0 < \gamma < T }} \frac{1}{\gamma}
+  \Big( x^{\sigma_1-1} - x^{-1/2} \Big)\sum_{\substack{ H_0 \le \gamma < T  }} \frac{1}{\gamma}.
\end{equation}
We conclude by applying Lemma \ref{b1Bnd} and Corollary \ref{prop:SOTO} to bound the above sums over the zeros:
\begin{equation}
\sum_{\substack{ 0 < \gamma < T }} \frac{1}{\gamma} \le S_0 + B_1(T_0, T)
,\ \text{ and }\ 
\sum_{\substack{ H_0 \le \gamma < T  }} \frac{1}{\gamma}\le B_1(H_0, T).
\end{equation}
\end{proof}

\subsection{Bounds for $\Sigma_{\sigma_1}^{\sigma_2}$}

In the following we shall use the notation
\begin{equation}\label{def-H0s}
H_{\sigma} = {\rm max}\left(H_0, \exp(1/R(1-\s))\right),
\end{equation}
which is defined so that we have
\[ N(\s,H_{\sigma}) = 0. \]
Thus we can rewrite \eqref{Sigmaab} as
\begin{equation}\label{Sigmaab1}
\Sigma_{a}^{b}= 2\sum_{\substack{H_a \le \gamma<T \\ a \le \beta <  b}} \frac{x^{\beta-1}}{\gamma}.
\end{equation}
We are going to split $\Sigma_{\sigma_1}^{\sigma_2}$ along the real axis.
We introduce the points $\s^{(n)}$ for $0 \le  n < N$:
\begin{equation}\label{def-sigma(n)}
\s^{(n)}  = \s_1+ (\s_2 - \s_1)\frac{n}{N},
\end{equation}
and the associated heights
\begin{equation}\label{def-H(n)}
H^{(n)} = {\rm max}(H_0, \exp(1/R(1-\s^{(n)}))) .
\end{equation}
\begin{Rem}
Note that for $\s<1-\frac1{R\log H_0}$, then $H_{\sigma} =  H_0$.
Here, $\frac1{R\log H_0} = 5.981 \cdot 10^{-14}$, 
so we will often have $H^{(n)} =H_0$ for the points $\s^{(n)}$ considered.
\end{Rem}
\begin{prop}\label{bnd-Sigma-sigma1-sigma2}
Let $N\ge2$ be an integer. 
If $ 5/8 \le  \s_1 < \s_2 \le  1$, $T \geq H_0$, then
\begin{equation}   
\Sigma_{\sigma_1}^{\sigma_2}  =2\sum_{\substack{0<\gamma<T \\ \sigma_1 \le \beta <  \sigma_2}} \frac{x^{\beta-1}}{\gamma} \le  \varepsilon_3(x,\s_1,\s_2,N,T) ,
\end{equation}
where
\begin{multline}\label{def-epsilon3}
\varepsilon_3(x,\s_1,\s_2,N,T)  = 2 x^{-(1-\s_1)+\frac{\s_2-\s_1}{N}}  B_0\big( \s_1, H_{\s_1}, T\big) 
\\+ 2x^{-(1-\s_1)} \big(1 - x^{-\frac{(\s_2-\s_1)}{N}}\big)  \sum_{n=1}^{N-1}B_0( \s^{(n)}, H^{(n)}, T) x^{ (\s_2-\s_1)\frac{n+1}{N}} ,
\end{multline}
where $B_0$ is defined in \eqref{def-B4}.
\end{prop}
Note that the assumption $\s_1\ge 5/8$ ensures to satisfy conditions of Theorem \ref{ZD-KLN}. 
\begin{proof}
We recall that
\[
\Sigma_{\sigma_1}^{\sigma_2}
=
2 \sum_{\substack{ H_0 \leq \gamma < T \\ \s_1 \leq \beta  <\s_2}} \frac{x^{\beta-1}}{\gamma} .
\]
We split along the real axis at the points $\s^{(n)}$, so as to use $ x^{\beta-1} < x^{\s^{(n+1)}-1}$ on each section:
\begin{equation}
\Sigma_{\sigma_1}^{\sigma_2}
\le  2 \sum_{n = 0}^{N-1} \sum_{\substack{ H^{(n)} < \gamma < T \\ \s^{(n)}\le \beta < \s^{(n+1)}}} \frac{x^{\s^{(n+1)}-1}}{\gamma} ,
\end{equation}
and we recognize the inner sums to be $s_0(\s^{(n)}, H^{(n)},T)$, where $s_0$ is defined in \eqref{def-s0}. \begin{equation}
\Sigma_{\sigma_1}^{\sigma_2}
\le  2 \sum_{n = 0}^{N-1}  x^{\s^{(n+1)}-1} \Big( s_0(\s^{(n)},H^{(n)},T)- s_0(\s^{(n+1)},H^{(n)},T) \Big).
\end{equation}
Instead of directly bounding the above with 
\[
2 \sum_{n = 0}^{N-1}  x^{\s^{(n+1)}-1} B_0(\s^{(n)},H^{(n)},T),
\]
we rearrange and telescope the terms: 
\begin{align*}
&2\sum_{n = 0}^{N-1}  x^{\s^{(n+1)}-1} \big( s_0(\s^{(n)},H^{(n)},T)- s_0(\s^{(n+1)},H^{(n)},T) \big)
\\& = 2x^{\s^{(1)}-1} s_0(\s^{(0)},H^{(0)},T)\! -\! 2x^{\s^{(N)}-1}  s_0(\s^{(N)},H^{(N)},T) +2\sum_{n = 1}^{N-1}\!  \big( x^{\s^{(n+1)}-1}\! -\! x^{\s^{(n)}-1}  \big)s_0(\s^{(n)},H^{(n)},T).
\end{align*}
Then we apply Lemma \ref{B4Bound} to bound each $s_0$ by $B_0$, and obtain the bound:
\[
2x^{\s^{(1)}-1} 
B_0(\s^{(0)},H^{(0)},T) 
+2\sum_{n = 1}^{N-1}  \Big( x^{\s^{(n+1)}-1} - x^{\s^{(n)}-1}  \Big)B_0(\s^{(n)},H^{(n)},T).
\]
We conclude by noting that
$x^{\s^{(n+1)}-1} - x^{\s^{(n)}-1} = x^{-(1-\s_1)}\Big( 1-x^{-\frac{(\s_2 - \s_1)}{N}} \Big)  x^{(\s_2 - \s_1)\frac{n+1}{N}}$ 
and that
$\s^{(1)}  = \s_1+ \frac{(\s_2 - \s_1)}{N}$.
\end{proof}

\begin{Rem}\ \noindent
\begin{enumerate}
\item
Taking $N=1$ and $T\to\infty$ gives a very simple bound, see Corollary \ref{cor:simpleboundSig12}.
Taking $N\to \infty$ simply gives a tighter numerical bound to a certain Riemann-Stieltjes integral. We illustrate the numerical benefit of increasing $N$ in Table \ref{table:comparisonSi12} as well as the asymptotic behaviour as $T\to \infty$.
\item 
In order to use the above numerically one needs results on zero densities for many values $\s$. These can be obtained on intervals as in Corollary \ref{ZDinterval}.
In the numerical  results we provide we use values that are optimized for each $\s$.
Though it is impractical to list the over $100\,000$ values we optimized, the total computation takes less than an hour and thus is reasonable if one wants optimal values.
\end{enumerate}
\end{Rem}
\begin{table}[t]
\caption{Comparison of bound $\varepsilon_3(x,\s_1,\s_2,N,T)$ as defined in \eqref{def-epsilon3} for $\Sigma_{\s_1}^{\s_2}$ using the values $T=\exp(c\sqrt{\log x/R})$, $\s_1=0.9$, $\s_2=0.993$.}
\label{table:comparisonSi12}
{
\begin{tabular}{|c|r|lllll|}
\hline
c & $\log x$ &  N=1 & N=10 & N=100  & N=1\,000 & N=10\,000 \\
\hline
1 & 5\,000 & 4.3062 e-19 & 4.9030 e-22 & 2.4879 e-22 & 2.3529 e-22 & 2.3442 e-22 \\
1 & 50\,000 & 1.2511 e-155 & 1.1789 e-158 & 5.8879 e-159 & 5.4937 e-159 & 5.4624 e-159 \\
\hline
2 & 5\,000 & 7.9441 e-19 & 7.4857 e-22 & 3.7416 e-22 & 3.5343 e-22 & 3.5211 e-22 \\
2 & 50\,000 & 1.2511 e-155 & 1.1789 e-158 & 5.8879 e-159 & 5.4937 e-159 & 5.4624 e-159 \\
\hline
16 & 5\,000 & 7.9441 e-19 & 7.4857 e-22 & 3.7416 e-22 & 3.5343 e-22 & 3.5211 e-22 \\
16 & 50\,000 & 1.2511 e-155 & 1.1789 e-158 & 5.8879 e-159 & 5.4937 e-159 & 5.4624 e-159 \\
\hline
\end{tabular}
}
\end{table}

\begin{cor}\label{cor:simpleboundSig12}
If $\sigma_1 \geq 0.9$ we have
\[  \Sigma_{\sigma_1}^{\sigma_2} \le    0.00125994\,x^{\sigma_2 - 1}.
 \]                                                              
\end{cor}
\begin{proof}
Taking $N=1$ in Prop. \ref{bnd-Sigma-sigma1-sigma2} we obtain the trivial estimate
\[ \Sigma_{\sigma_1}^{\sigma_2} < 2B_0( \s_1, H_{\s_1}, T)x^{\s_2-1}. \]
Noting in \eqref{def-B4} that $\lim_{T\to\infty}\Gamma(s,T) =0$, for all $\ s>1$, then 
\begin{equation} 
\lim_{T\to \infty}  B_0( \s_1, H_{\s_1}, T) = 
\frac{c_1}{(1-p)^{q+1} }  \Gamma(q+1,(1-p)\log H_{\s_1}) 
+ c_2  \Gamma(3, \log H_{\s_1}),
 \end{equation} 
giving 
\[ \lim_{T\to \infty} B_0( 0.9, 3\cdot 10^{12}, T) \le  
0.000629970 .\qedhere\]
\end{proof}
\subsection{Bounds for $\Sigma_{\sigma_2}^{1}$}
Thanks to the zero free region \eqref{zfr}, the remaining zeros $\varrho=\beta+i\gamma$ satisfying $ \beta \ge \s_2$ and $0< \gamma <T$, are actually contained in the region 
$ \beta \le 1-\frac1{R\log \gamma} $ and $H_{\sigma_2} \le \gamma <T$, 
where $H_{\sigma}$ is defined in \eqref{def-H0s}.
Thus we have the bound for $\Sigma_{\s_2}^{1}$, as defined in \eqref{Sigmaab}:
\begin{equation}\label{Sigma2-1}
\Sigma_{\s_2}^{1} 
= 2 \sum_{\substack{ H_{\sigma_2}\le  \gamma <T \\  \s_2 \le \beta \le 1-\frac1{R\log \gamma}}} \frac{x^{\beta-1}}{\gamma}
\le  2 \sum_{\substack{ H_{\sigma_2}\le  \gamma <T \\  \beta \ge \s_2  }} \frac{x^{-\frac1{R\log \gamma}}}{\gamma}.
\end{equation}
Following the approach taken in \cite[Section 3]{PlaTru21ET} we apply Riemann-Stieltjes integration to obtain the following.
\begin{prop}\label{prop:eps4bound}
Let $ 5/8  < \s_2 \le  1$,  $t_0 =t_0(\s_2,x)= {\rm max} \big( H_{\sigma_2}, \exp\big( \sqrt{\frac{ \log x}R} \big) \big)$ and $T> t_0$.
Let $K \geq 2$ and consider $(t_k)_{k=0}^{K}$ a strictly increasing sequence such that $t_K = T$. 
 Then
\[
\Sigma_{\sigma_2}^{1}  = 2\sum_{\substack{0<\gamma<T \\ \sigma_2 \le \beta <  1}} \frac{x^{\beta-1}}{\gamma} \le   2 N(\sigma_2,T)  \frac{x^{-\frac{1}{R\log t_0 }}}{t_0} ,
\]
and 
\[
\Sigma_{\s_2}^{1}
\le  2 \Bigg( \Big( \sum_{k = 1}^{K-1} {N}(\s_2, t_k)\Big(  \frac{x^{-\frac{1}{R\log(t_{k-1})}}}{t_{k-1}} - \frac{x^{-\frac{1}{R\log t_{k} }}}{t_k} \Big)\Big)+ \frac{x^{-\frac{1}{R\log(t_{K-1})}}}{t_{K-1}}{N}(\s_2, T) \Bigg). 
\]
\end{prop}
The next corollary fixes how we choose the sequence $(t_k)$.
\begin{cor}\label{Cor-bnd-eps4}
Let $ 5/8  < \s_2 \le  1$,  $t_0 =t_0(\s_2,x)= {\rm max} \big( H_{\sigma_2}, \exp\big( \sqrt{\frac{ \log x}R} \big) \big)$, $T> t_0$.
Let $K \geq 2$, $\lambda = (T/t_0)^{1/K}$, and consider $(t_k)_{k=0}^{K}$ the sequence given by $t_k = t_0\lambda^k$.
Then 
\[
\Sigma_{\s_2}^{1} = 2\sum_{\substack{0<\gamma<T \\ \sigma_2 \le \beta <  1}} \frac{x^{\beta-1}}{\gamma} \
\le  \varepsilon_4(x,\s_2,K,T) ,
\]
where  
\begin{equation}\label{def-epsilon4} 
\varepsilon_4(x,\s_2,K,T)  = 2 \sum_{k=1}^{K-1}  \frac{x^{-\frac{1}{R\log t_{k} }}}{t_k} \Big(  \tilde{N}(\s_2, t_{k+1}) - \tilde{N}(\s_2, t_{k}) \Big)  + 2  \tilde{N}(\s_2, t_{1})\frac{x^{-\frac{1}{R(\log t_{0})}}}{t_0}  ,
\end{equation}
and $\tilde{N}(\sigma,T)$ satisfy \hyperref[ZDB]{(ZDB)} $N(\sigma,T)\le \tilde{N}(\sigma,T)$. 
\end{cor}
\begin{proof}[Proof of Proposition \ref{prop:eps4bound}]
Classical Riemann-Stieltjes integration followed by an integration by part give:
\begin{equation}\label{ineq1}
\Sigma_{\sigma_2}^{1}  \le 2 \int_{H_{\sigma_2}}^T  \frac{x^{-\frac{1}{R\log t }}}{t} dN(\sigma_2,t)
= 2\Big( \frac{x^{-\frac{1}{R\log T }}}{T}  N(\sigma_2,T) - \int_{H_{\sigma_2}}^T N(\sigma_2,t) {d}\Big( \frac{x^{-\frac{1}{R\log t}}}{t} \Big) \Big).
 \end{equation}
Note that
\[  \frac{d}{dt}\Big( \frac{x^{-\frac{1}{R\log t }}}{t} \Big) = \frac{x^{-\frac{1}{R\log t }}}{t^2} \Big(\frac{\log x}{R(\log t)^2} -1 \Big)  \]
becomes negative as soon as $t> \exp\big( \sqrt{\frac{ \log x}R} \big) $.
Thus taking $t_0={\rm max} \big( H_{\sigma_2}, \exp\big( \sqrt{\frac{ \log x}R} \big) \big)$, we have the bound 
\[
 - \int_{H_{\sigma_2}}^T N(\sigma_2,t) {d}\Big( \frac{x^{-\frac{1}{R\log t}}}{t} \Big) 
 \le - N(\sigma_2,T)   \int_{t_0}^T {d}\Big( \frac{x^{-\frac{1}{R\log t }}}{t} \Big)
 =  N(\sigma_2,T)  \Big( \frac{x^{-\frac{1}{R\log t_0 }}}{t_0} - \frac{x^{-\frac{1}{R\log T }}}{T}\Big) .
\]
Combining this with \eqref{ineq1}, we obtain
\[ 
\Sigma_{\sigma_2}^{1}   \le   2 N(\sigma_2,T)  \frac{x^{-\frac{1}{R\log t_0 }}}{t_0} .
\] 
We refine this bound by splitting along the imaginary axis at the heights $t_k$, taking $t_0$ as above and $t_K = T$.  
We can use the bound $N(\sigma_2,t) \le N(\sigma_2,t_{k+1})$ on each piece, so that
\[
- \int_{H_{\sigma_2}}^T N(\sigma_2,t) \frac{d}{dt}\Big( \frac{x^{-\frac{1}{R\log t}}}{t} \Big) 
\le \sum_{k=0}^{K-1} 
N(\sigma_2,t_{k+1}) \Big( \frac{x^{-\frac{1}{R\log t_k}}}{t_k} - \frac{x^{-\frac{1}{R\log t_{k+1}}}}{t_{k+1}} \Big) .
\]
We note that the $(K-1)^{th}$ term is $N(\sigma_2,T) \Big( \frac{x^{-\frac{1}{R\log t_{K-1}}}}{t_{K-1}} - \frac{x^{-\frac{1}{R\log T}}}{T} \Big)  $, so that when combining the above with \eqref{ineq1}, and shifting the sum, we obtain
\[
\varepsilon_4(x,\s_2,K,T)  =  2 \Big( \sum_{k = 1}^{K-1} \tilde{N}(\s_2, t_k)\Big(  \frac{x^{-\frac{1}{R\log(t_{k-1})}}}{t_{k-1}} - \frac{x^{-\frac{1}{R\log t_{k} }}}{t_k} \Big)+ \frac{x^{-\frac{1}{R\log(t_{K-1})}}}{t_{K-1}}\tilde{N}(\s_2, T) \Big).
\]
We conclude by `telescoping' the sum over $k$.
\end{proof}

\begin{Rem}
The parameters $T$ and $\s_2$ play a minor role in the main asymptotic of $\Sigma_{\sigma_2}^1$:
\begin{enumerate}
\item 
With $T$ fixed, as $K$ increases, the bound $\varepsilon_4(x,\s_2,K,T) $ as defined in \eqref{def-epsilon4}  provides a very good numerical approximation of the integral expression 
\[2 \int_{t_0}^T  \frac{x^{-\frac{1}{R\log t }}}{t} d\tilde N(\sigma_2,t).
\]
The use of `telescoping' is crucial in obtaining a good numerical  approximation of the Stieltjes integral.
Table \ref{table:epsilon4} illustrates the convergence of $\varepsilon_4(x,\s_2,K,T) $ as $K$ increases.
The tightness of this bound is important in numerical  results as well as reducing the constant $A$ in our asymptotic results. Previous work often bounds the density of zeros by a constant. 
\item
By fixing $\lambda$ and considering $T = t_0\lambda^K$ as a function of $K$, 
Table \ref{table:epsilon4} suggests that the bound $\varepsilon_4(x,\s_2,K,t_0 \lambda^K) $ converges as $K\to \infty$.
We note that this would provide a tool to obtain both accurate numerical and asymptotic bounds for 
\[
\sum_{\substack{ H_0 \le \gamma \\ \s_2 \le \beta <1}}  \frac{x^{\beta-1}}{\gamma}  . \]
To establish this, one would need to 
perform an analysis as in Proposition \ref{epsilon4-dec}.
As we will not use this we omit the details for now.
\item
It is also worth noticing that with $T > t_0$ one easily obtains a lower bound for $\Sigma_{\s_2}^{1}$ of $ \frac{2}{t_0^2}  N(\sigma,t_0) $
which, if we assume our current \hyperref[ZDB]{(ZDB)}, is bounded below by a constant times $\frac{(\log t_0)^{3}}{t_0^2} $
which is asymptotically of the same order of magnitude as the upper bound from $\varepsilon_4$.
Hence Corollary \ref{Cor-bnd-eps4} describes the true asymptotic of this method with our current tools.
Significant improvements therefore require better zero-free regions, better zero densities, or the use of a more sophisticated weight, as in the work of B\"uthe \cite{Bu16}.
\end{enumerate}
\end{Rem}

\begin{table}
\caption{Comparison of $\varepsilon_4(x,\s_2,K,T)$ as defined in \eqref{def-epsilon4} with $T=t_0^c$ for various values of $c$ and $x$}\label{table:epsilon4}.
{
\begin{tabular}{|c|r|lllll|}
\hline
c & $\log (x)$ &  K=1 & K=10 & K=100  & K=1\,000 & K=10\,000 \\
\hline
2 & 5\,000 & 2.2105 e-19 & 3.4577 e-20 & 3.0121 e-20 & 2.9709 e-20 & 2.9669 e-20  \\
2 & 50\,000 & 4.1123 e-73 & 1.7519 e-74 & 1.4268 e-74 & 1.3982 e-74 & 1.3954 e-74  \\
\hline
4 & 5\,000 & 5.4637 e-18 & 4.7133 e-20 & 3.1058 e-20 & 2.9801 e-20 & 2.9678 e-20  \\
4 & 50\,000 & 1.1429 e-70 & 3.2640 e-74 & 1.4928 e-74 & 1.4045 e-74 & 1.3960 e-74  \\
\hline
8 & 5\,000 & 4.1362 e-16 & 1.1462 e-19 & 3.3022 e-20 & 2.9983 e-20 & 2.9696 e-20  \\
8 & 50\,000 & 1.0929 e-66 & 1.4821 e-73 & 1.6355 e-74 & 1.4172 e-74 & 1.3973 e-74  \\
\hline
16 & 5\,000 & 2.9350 e-13 & 5.7273 e-19 & 3.7327 e-20 & 3.0353 e-20 & 2.9732 e-20  \\
16 & 50\,000 & 1.2372 e-59 & 1.9513 e-72 & 1.9679 e-74 & 1.4430 e-74 & 1.3998 e-74  \\
\hline
\end{tabular}
}

\end{table}

To make use of this bound we must establish that it is decreasing in $x$.
\begin{prop}\label{epsilon4-dec}\label{epsilon4-dec-simple}
Fix $K\ge 2$ and $c>1$, and set $t_0, T$, and $\sigma_2$ as functions of $x$ defined by 
\begin{equation}\label{def-t0-T-sigma2}
t_0 = t_0(x) = \exp\Big(\sqrt{\frac{\log x} R}\Big),\ 
T=t_0^c,\ \text{and }\ 
\sigma_2 = 1-\frac{2}{R\log t_0} .
\end{equation}
Then, with $\varepsilon_4(x,\s_2,K,T)$ as defined in \eqref{def-epsilon4}, we have that as $x\to \infty$,
\begin{equation}\label{asymp-eps4}
\varepsilon_4(x,\s_2,K,T) 
=(1+o(1)) C \frac{(\log t_0)^{3+\frac{4}{R\log t_0}}}{t_0^2}, \ \text{with}\ 
  C = 2 c_1   e^{\frac{16w_{1}}{3 R} } w_{1}^{3},
  \ \text{and}\ w_1 =1 + \frac{c-1}{K},
\end{equation}
where $c_1$ is an admissible value for \hyperref[ZDB]{(ZDB)} on some interval $[\sigma_1,1]$.
Moreover, both $\varepsilon_4(x,\s_2,K,T)$ and  
$   \frac{\varepsilon_4(x,\s_2,K,T)t_0^2}{(\log t_0)^3} $
are decreasing in $x$ for $x>\exp(R e^2)$. 
\end{prop}
\begin{proof}
Denoting $w_k = 1 + \frac{k}{K}(c-1)$, with $w_0=1$, and recalling $t_k = t_0\lambda^k$, we have the identities 
\begin{equation}\label{def-tk}
\begin{split}
& \log x = R(\log t_0)^2, \ 
\log t_0 = \sqrt{\frac{\log x}{R}}, \ 
\log t_k  
= (\log t_0) w_k, \\ 
\text{ so }\ 
& \frac{x^{-\frac{1}{R(\log t_{0})}} }{t_0} = t_0^{-2}, \ 
\frac{x^{-\frac{1}{R\log t_{k} }} }{t_k} 
=  e^{-(\log t_0) \big(\frac{1}{w_k} + w_k\big)} 
= t_0^{-2 - \frac{(w_k-1)^2}{w_k} } .
\end{split}
\end{equation}
This allows us to rewrite \eqref{def-epsilon4} as
\begin{equation}\label{eps4-1}
\varepsilon_4(x,\s_2,K,T) =\frac{2}{t_0^2}\sum_{k=1}^{K-1} t_0^{ - \frac{(w_k-1)^2}{w_k} } 
\Big(  \tilde{N}(\s_2, t_{k+1}) - \tilde{N}(\s_2, t_{k}) \Big)  
+\frac{2}{t_0^2} \tilde{N}(\s_2, t_{1}).
\end{equation}
We now focus on the expression for $\tilde N$ as given by \eqref{ZD}.
Since
\[
\sigma_2 = 1-\frac{2}{R\log t_0} ,\
\ p(\s_2) = \frac{8}{3}(1-\s_2) = \frac{16}{3 R\log t_0}, 
\  q(\s_2) = 5-2\s_2= 3+\frac{4}{R\log t_0},
\]
then
\begin{multline}\label{diffN}   
\tilde{N}(\s_2, t_{k+1}) - \tilde{N}(\s_2, t_{k}) 
= 
c_1  t_{k+1}^{\frac{16}{3 R\log t_0}} (\log  t_{k+1} )^{3+\frac{4}{R\log t_0}} - c_1  t_{k}^{\frac{16}{3 R\log t_0}} (\log  t_{k} )^{3+\frac{4}{R\log t_0}}
\\ + c_2  (\log  t_{k+1})^2 - c_2  (\log  t_{k})^2 .
\end{multline}
Using \eqref{def-tk}, we can simplify the above expressions: 
\begin{multline}
 c_1  t_{k+1}^{\frac{16}{3 R\log t_0}} (\log  t_{k+1} )^{3+\frac{4}{R\log t_0}} - c_1  t_{k}^{\frac{16}{3 R\log t_0}} (\log  t_{k} )^{3+\frac{4}{R\log t_0}}
\\= c_1 (\log t_0)^{3+\frac{4}{R\log t_0}}  \Big( e^{\frac{16w_{k+1}}{3 R} } w_{k+1}^{3+\frac{4}{R\log t_0}} -  e^{\frac{16w_k}{3 R } }  w_k ^{3+\frac{4}{R\log t_0}}\Big) ,
\end{multline}
and
\begin{equation}
c_2  (\log  t_{k+1})^2 - c_2  (\log  t_{k})^2 
=  c_2  (\log  t_0)^2 \big(w_{k+1}^2 - w_{k}^2\big) .
\end{equation}
Thus \eqref{diffN} can be rewritten as  
\begin{equation}\label{diffN1}
\tilde{N}(\s_2, t_{k+1}) - \tilde{N}(\s_2, t_{k}) 
= 
C_{1,k} (\log t_0)^{3+\frac{4}{R\log t_0}}
+ C_{2,k}  (\log t_0)^{2},
\end{equation}
with
\begin{equation}
C_{1,k} = c_1   \Big( e^{\frac{16w_{k+1}}{3 R} } w_{k+1}^{3+\frac{4}{R\log t_0}} -  e^{\frac{16w_k}{3 R } }  w_k ^{3+\frac{4}{R\log t_0}}\Big) ,\  
C_{2,k}= c_2 \big(w_{k+1}^2 - w_{k}^2\big).
\end{equation}
In addition we have
\[
\tilde{N}(\s_2, t_{1}) 
= 
c_1e^{\frac{16w_{1} }{3 R} }w_1^{3+\frac{4}{R\log t_0}}  (\log t_0)^{3+\frac{4}{R\log t_0}} 
+ c_2 w_{1}^2 (\log  t_0)^2 .
\]
Together with \eqref{diffN1}, we deduce that \eqref{eps4-1} can be simplified to
\begin{multline}\label{eps4-2}
\varepsilon_4(x,\s_2,K,T) =
 \frac{2(\log t_0)^3}{t_0^2} \sum_{k=1}^{K-1} t_0^{ - \frac{(w_k-1)^2}{w_k} } 
\Big( C_{1,k} (\log t_0)^{\frac{4}{R\log t_0}}
+ \frac{C_{2,k}}{  (\log t_0) } \Big)  
\\ +\frac{2(\log t_0)^3}{t_0^2} \left(c_1 (\log t_0)^{ \frac{4}{R\log t_0}} e^{\frac{16w_{1}}{3 R} } w_{1}^{3+\frac{4}{R\log t_0}}
+ \frac{ c_2w_{1}^2 }{  (\log t_0) }  
\right).
\end{multline}
Since $ t_0^{ - \frac{(w_k-1)^2}{w_k} } , (\log t_0)^{\frac{4}{R \log t_0}} , \frac1{\log t_0}$, and $w_{1}^{3+\frac{4}{R\log t_0}}$ all decrease with $t_0>e^e$, then the expression for 
$\frac{\varepsilon_4(x,\s_2,K,T) t_0^2}{(\log t_0)^3}$ as well as just $ \varepsilon_4(x,\s_2,K,T) $ both decrease with $t_0$, and thus with $x$.
We can also deduce that
\[
\frac{\varepsilon_4(x,\s_2,K,T) t_0^2}{(\log t_0)^{3+\frac{4}{R\log t_0}}}=
 2   \sum_{k=1}^{K-1} t_0^{ - \frac{(w_k-1)^2}{w_k} } 
\Big( C_{1,k} 
+\frac{ C_{2,k} }{ (\log t_0)^{1+\frac{4}{R\log t_0}} }\Big)  
 +2\Big(c_1   e^{\frac{16w_{1}}{3 R} } w_{1}^{3+\frac{4}{R\log t_0}}
+  \frac{c_2w_{1}^2 }{ (\log t_0)^{1+\frac{4}{R\log t_0}} }\Big).
\]
Letting $x\to \infty$, this gives the asymptotic announced in \eqref{asymp-eps4}.
\end{proof}

\begin{Rem}
From the proposition we see that as $x\to\infty$ the value of $\varepsilon_4(x,\s_2,K,T) $ will tend to a limit value of $C = 2 c_1   e^{\frac{16w_{1}}{3 R} } w_{1}^{3}\approx90.825$, where the parameters values are chosen as in Section \ref{sec:asymptotic}.
Some of the asymptotic behaviour can be seen in the values of $A(x_0)$ in Table \ref{TableResultsAsympt}.
\end{Rem}

\section{Proof of Theorem \ref{thm-numericVersion}}\label{sec:numeric}

\begin{proof}[Proof of Theorem \ref{thm-numericVersion}]
The verification \cite{PlaTru21RH} and the zero-free region \cite{TrudMos15} allow to take \[ H_0 = 3\cdot 10^{12},\quad T_0=30\,610\,046\,000\ \text{and}\ R= 5.5666305.\] 
We now apply Proposition \ref{Perron-exp} together with Proposition \ref{bnd-Sigma-0-sigma1}, Proposition \ref{bnd-Sigma-sigma1-sigma2}, and Proposition \ref{prop:eps4bound}.
By Proposition \ref{epsilon4-dec-simple}, $\varepsilon_4$ is decreasing in $x$, and 
this is also easily seen for each of $\varepsilon_1, \varepsilon_2,$ and $\varepsilon_3$. 
Taking 
\[T=t_0(\sigma_2,x_0)^c \ \text{and}\ 
\s_1 = 0.9,
\] we have then that for all $x\ge x_0$,
\begin{equation}
\label{def-epsilon} 
E_{\psi}(x) \le  \varepsilon(x_0,\s_2,c, N, K) := \varepsilon_1(x_0,T) + \varepsilon_2(x_0,0.9,T_0,T) + \varepsilon_3(x_0,0.9,\s_2,N,T) + \varepsilon_4(x_0,\sigma_2,K,T). 
\end{equation}
To produce Table  \ref{TableResultsNumeric} we shall fix
\begin{equation}\label{vals-param}
c=3, \ K = 100\,000, \text{and}\ N=\lceil100\,000\,(\s_2-\s_1) \rceil,
\end{equation}
and attempt to optimize the choice of $\s_2$.
We have used an optimized version of the zero density results Theorem \ref{ZD-KLN} for each $\sigma$ considered.
Each row of Table \ref{TableResultsNumeric} reports a numerical computation of each $\varepsilon_i$ with the given choices of parameters.
\end{proof}

\begin{Rem}
The choice of $c=3$ was based on some experimentation, for larger values $x$ (say $x>\exp(10\,000)$) a smaller value might improve the result and for smaller values $x$ (say $x<\exp(2000)$) a larger value might improve the result. 
In both cases the observed improvement was not in the first few decimal places.

The choice $N=\lceil100\,000(\s_2-\s_1)\rceil$ is entirely for computational reasons, it ensures the collection of $\s$ for which we need to optimize the zero density results don't depend on the actual choice of $\s_2$.
As illustrated in Tables \ref{table:comparisonSi12} and \ref{table:epsilon4} reducing $N$ and $K$ does not significantly impact the final result.
\end{Rem}

\section{Proof of Theorem \ref{thm-asymptoticVersion}}\label{sec:asymptotic}

\begin{proof}[Proof of Theorem \ref{thm-asymptoticVersion}]
As in Section \ref{sec:numeric}, we immediately have that
\[ E_{\psi}(x) \le  \varepsilon_1(x,T) + \varepsilon_2(x,\sigma_1,T_0,T) +  \Big| \Sigma_{\s_1}^{\s_2}  \Big| +   \varepsilon_4(x,\sigma_2,K,T) .\]
However we shall use simpler bounds for $ \Sigma_{\s_1}^{\s_2} $.
We shall take
\begin{equation}\label{Thm1.2-parameters}
\begin{split}
& t_0 =  t_0(x) = \exp( \sqrt{\log x/R} ) , \quad c= 4, \quad 
T = T(x) =  t_0^4 =\exp(4 \sqrt{\log x/R} ) ,\ \qquad\ \\ 
& \s_1 = 0.9 ,\ \qquad\ 
\text{and}\ \qquad\ 
 \s_2 = \s_2(x) = 1 - \frac{2}{R\log t_0 } = 1 - \frac{2}{\sqrt{R\log x}}.  
\end{split}
\end{equation}
With these choices we may assume that in the definition of $ \varepsilon_4(x,\sigma_2,K,T) $ all values of $\sigma_2$ lie in $ [0.9,1]$. Thus we 
may use the values from Table \ref{table:NSigmaOptTable2}:
\[c_1 = 17.4194  \ \text{and}\ c_2 = 2.9089.\] 

We immediately obtain the following asymptotics:
\begin{enumerate}

\item Using the $\varepsilon_1$ of Proposition \ref{Perron-exp}, we have  $2\frac{(\log x)^2}{T} = 2R^2 \frac{\log(t_0)^4}{t_0^4} $
and for $\log x>1\,000$ this is less than
\[ 1.92998\cdot 10^{-9}  \frac{(\log t_0)^3}{t_0^2}  .\]

\item We have $\varepsilon_2(x,0.9,T_0,T) =  2 x^{-1/2} (S_0 + B_1(T_0,T))
+ \Big( x^{\sigma_1-1} - x^{-1/2} \Big)B_1(H_0, T)$ and for $\log x > 1\,000$ this is less than
\[   10^{-10}    \frac{(\log t_0)^3}{t_0^2} . \]

\item By Corollary \ref{cor:simpleboundSig12} we have $\Big| \Sigma_{0.9}^{\s_2}  \Big|  <  0.00125994 x^{-\s_2} = \frac{ 0.00125994}{t_0^2}$ and for $\log x>1\,000$  this is less than
\[ 6.8979 \cdot 10^{-6} \frac{(\log t_0)^3}{t_0^2} .\]

\item We have $  \varepsilon_4(x,\sigma_2,K,T)   <   \frac{  \varepsilon_4(x_0,\sigma_2,K,T) t_0^2}{(\log t_0)^3}    \frac{(\log t_0)^3}{t_0^2} $
and we recall that by Proposition \ref{epsilon4-dec} the quantity
\[   \frac{  \varepsilon_4(x,\sigma_2,K,T) t_0^2}{(\log t_0)^3}  \]
is decreasing in $x$.
\end{enumerate}
It follows from the above that for $x>x_0$ we have
\[   E_{\psi}(x)  <    A(x_0) \frac{(\log t_0)^3}{t_0^2} ,\]
where 
\begin{equation} \label{def-Ax} A(x_0) =  \frac{  \varepsilon_4(x_0,\sigma_2,K,T(x_0)) (t_0(x_0))^2}{(\log (t_0(x_0)))^3}  + 1.92998\cdot 10^{-9} + 10^{-10} +  6.8979 \cdot 10^{-6}. \qedhere\end{equation}
\end{proof}

\begin{Rem}
We note that the formula $A(x_0)$ depends on $R$, however it is decreasing in $R$, and so each line of Table \ref{TableResultsAsympt} gives a value of $A(x_0)$ which is admissible for  $R = 5.5666305$.

In the theorem we verify that the formula for $A(x_0)$ is valid for all $R<5.573412$, however we see that upper bound could be relaxed considerably.

 The following table illustrates the effect of using different values of $R$ (note that only the first two are known to be valid).
Note that although $A(x_0)$ gets worse as $R$ is decreased, the actual values of $\varepsilon$ decrease.
\begin{center}
\begin{tabular}{|c|r|c|c|c|}
\hline
$R_0$ & $\log(x_0)$ & $A(x_0)$ & $\varepsilon_{\text{asm}}(x_0,x_0)$ & $\varepsilon(x_0)$ \\
\hline
$5.57341200$ & $2\,500$ &  245.3670 & 9.3655 e-13 & 1.1173 e-13 \\
$5.57341200$ & $5\,000$ &  204.0403 & 5.2857 e-20 & 3.0972 e-20 \\
$5.57341200$ & $50\,000$ &  138.1490 & 6.3117 e-75 & 4.2087 e-75 \\
\hline
$5.56663050$ & $2\,500$ &  245.6773 & 9.1553 e-13 & 1.0975 e-13 \\
$5.56663050$ & $5\,000$ &  204.2929 & 5.1120 e-20 & 2.9943 e-20 \\
$5.56663050$ & $50\,000$ &  138.3136 & 5.6411 e-75 & 3.7604 e-75 \\
\hline
$5.50000000$ & $2\,500$ &  248.7907 & 7.3084 e-13 & 9.1632 e-14 \\
$5.50000000$ & $5\,000$ &  206.8264 & 3.6693 e-20 & 2.1427 e-20 \\
$5.50000000$ & $50\,000$ &  139.9638 & 1.8502 e-75 & 1.2242 e-75 \\
\hline
\end{tabular}
\end{center}
\end{Rem}

\begin{lem}\label{lem:butinterp}
For all $0 < \log x \le 2\,100$ we have that
\[    E_{\psi}(x)  \le    2 (\log x)^{3/2}\exp(-0.8476836\sqrt{\log x}).  \]
\end{lem}
\begin{proof}
For $0 < \log x\le  4$ the bound we are claiming is trivial.
For $4 < \log x \le 40$ the results of \cite{Bu18} give a much stronger bound.
For $20<\log x \le  2\,100$ the results arising from the work of \cite{Bu16} as calculated in \cite[Table 8]{BKLNW21} give much strong bounds than those claimed here.
Denoting by $\varepsilon_{\text{B\"u}}(x_0)$ bounds for $E_{\psi}(x)$ with $x>x_0$ obtained from \cite[Table 8]{BKLNW21} we have
\begin{align*} 
\varepsilon_{\text{B\"u}}(\exp(40))&=1.93378\cdot 10^{-8} <  2 (1\,000)^{3/2}\exp(-0.8476836\sqrt{1\,000}) ,\\
\text{ and }\ \varepsilon_{\text{B\"u}}(\exp(1\,000))&=1.94751\cdot 10^{-12} <  2 (2\,100)^{3/2}\exp(-0.8476836\sqrt{2\,100}).
\end{align*}
Using that $2 (\log x)^{3/2}\exp(-0.8476836\sqrt{\log x})$ is decreasing for $x>\exp(40)$  completes the result.
\end{proof}
\begin{lem}\label{lem:fksinterp1}
For all $2\,100 < \log x \le 200\,000$ we have that
\[    E_{\psi}(x)  \le    9.22022 (\log x)^{3/2}\exp(-0.8476836\sqrt{\log x}).  \]
\end{lem}
\begin{proof}
As in the previous lemma we use that the function we are comparing to is decreasing and interpolate a lower bound as a step function using explicit numerical  results.

We shall consider a collection of values $b_i$ which divide the interval $[2\,100,200\,000]$ and verify that
\[   
121.096 \left(\frac{b_{i+1}}{R}\right)^{3/2}\exp\left(- \frac{2}{\sqrt{R}} \sqrt{b_{i+1}}\right) > \varepsilon(\exp(b_i)) .
\]
In order to do in a way which exposes the limit of this method we instead tabulate values
\begin{equation}\label{eqaa}  
\varepsilon( \exp(b_i) )  \left(\frac{R}{b_{i}}\right)^{3/2}\exp\left( \frac{2}{\sqrt{R}} \sqrt{b_{i}}\right)    ,
\end{equation}
and bound the ratio 
\begin{equation}\label{eqbb} \left(\frac{R}{b_{i+1}}\right)^{\frac{3}{2}}\exp\left( \frac{2}{\sqrt{R}}  \sqrt{b_{i+1}}\right) \left(\frac{b_{i}}{R}\right)^{\frac{3}{2}}\exp\left( \frac{-2}{\sqrt{R}}  \sqrt{b_{i}}\right)  = \left(\frac{b_i}{b_{i+1}}\right)^{3/2}\exp\left(\frac{2}{\sqrt{R}}(\sqrt{b_{i+1}}-\sqrt{b_{i}})\right)
\end{equation}
directly. The subdivision we use, the step size for each subdivision, and the maximum encountered in each interval are given in the following table.
The final column gives a lower bound for an $A$ which would be admissible on this interval and is the rounded up product of the two maximum values.

\begin{center}
\begin{tabular}{|c|c|rl|l|r|}
\hline
Interval & Step Size & Max Eq. \eqref{eqaa} & at ($ \log x $) & Max Eq. \eqref{eqbb} &  Bound on $A$\\
\hline
$2\,100 \leq  \log x  < 6\,000$ & $1/2$ &  $120.400$ &$(5\,999.5)$ & $1.0050$ &  $121.002$  \\ 
$6\,000 \leq  \log x  < 7\,250$ & $1/5$ & $120.894$  &$(7\,249.8)$ & $1.0011$ & $121.027$  \\ 
$7\,250 \leq  \log x  < 8\,250$ & $1/10$ & $121.035$ &$(7\,914.3)$ & $1.0005$ & $121.096$ \\ 
$8\,250 \leq  \log x  < 9\,000$  & $1/5$ & $120.805$ &$(8\,250)$  & $1.0009$ &  $120.913$ \\  
$9\,000 \leq  \log x  < 12\,000$  & $1$ &  $119.416$ &$(9\,000)$  & $1.0050$ &  $120.014$ \\  
$12\,000 \leq  \log x  < 100\,000$  & $5$ & $113.332$ &$(12\,000)$   & $1.0190$ &  $115.486$ \\ 
$100\,000 \leq  \log x  < 200\,000$  & $100$ & $84.520$ &$(100\,000)$  & $1.1420$ &  $96.522$ \\ 
\hline
\end{tabular}
\end{center}
We now simply verify that $121.096/5.5666305^{3/2} < 9.22022$ to complete the result. 
\end{proof}
\begin{Rem}
The fourth column of Table \ref{TableResultsAsympt} illustrates a sampling of the values  $ \varepsilon(x_0)  \tfrac{t_0(x_0)^2}{(\log t_0(x_0))^3}$ that we used for Lemma \ref{lem:fksinterp1}, these represent a theoretical limit to the interpolation.
These values increase initially, especially when $b_i<4\,595$ (so that $t_0<3\cdot10^{12}$). For  $4\,595 <b_i < 7914.3$ they slowly increase to the peak of $121.035$ which occurs at $b_i=7914.3$.
 Note that at $7914.3$ the optimal of $\sigma_2$ is very close to $1-\frac{1}{R\log(3\cdot10^{12})}$ which explains why there is a transition near this point.
 Beyond this point the values decrease.

One would certainly expect that by refining the intervals used one can bring the result down slightly. The computations required to verify all of the values used in the proof of Lemma \ref{lem:fksinterp1} already took over $4000$ cpu/hours on various machines in the Westgrid computing cluster.
\end{Rem}

\begin{proof}[Proof of Corollary \ref{cor:allx} ]
The right hand side of the formula
\[    E_{\psi}(x)  \le  9.22022(\log x)^{3/2}\exp(-0.8476836\sqrt{\log x}) \]
comes from  Lemma \ref{lem:fksinterp1}.

For all $\log x > 200\,000$ Theorem \ref{thm-asymptoticVersion} gives the result. 
For $200\,000 \ge \log x \ge 2\,100$ Lemma \ref{lem:fksinterp1} gives the result.
For $2\,100 \ge \log x \ge 0$ Lemma \ref{lem:butinterp} gives the result.
\end{proof}

\begin{Rem}\label{rem:lowx}
The bound 
\begin{equation}
 \label{bnd:smallx} E_{\psi}(x) \le  \varepsilon_s(x) :=  \varepsilon_1(x,T) + 2(S_0 + B_1(T_0,T))x^{-1/2}
 \end{equation}
is valid for any $S_0,T_0$ from Table \ref{tab:Sig1} and $T<3\cdot 10^{12}$.
In Table \ref{tab:smallx} we illustrate the relative effectiveness of this method for this region using the various $\varepsilon_1$ for different versions of Perron's formula admissible for this region (see  Remark \ref{rem:perron}).
Note that for $T$ fixed the formula is not decreasing in $x$, and as such, provides only a bound at $x$.
The values $\varepsilon_{\text{B\"u}}(x)$ come from the approach described in \cite{Bu16} as implemented for \cite{BKLNW21}.
\end{Rem}
\begin{table}[h!]
\centering
\caption{Bounds on $ E_{\psi}(x) $ for small $x$ using Equation \eqref{bnd:smallx} compared to methods of B\"uthe, \cite{Bu16}, and Perron's formula from Dudek and Cully-Hugill and Johnston, see \cite{Dud16} and \cite{CulJoh21}. See also Remark \ref{rem:perron} (we denote the error terms found with their respective methods $\varepsilon_{\text{Du}}, \varepsilon_{\text{CJ}}, \varepsilon_{\text{B\"u}}$).
Note, these bounds based on Perron's formula would hold only at $x$, see Remark \ref{rem:lowx}.}\label{tab:smallx}
\begin{tabular}{|r|l|l|l|}
\hline
$\log x$ & $\quad\varepsilon_{\text{Du}}(x)$ & \quad$\varepsilon_{\text{CJ}}(x)$ & \quad$\varepsilon_{\text{B\"u}}(x)$ \\
\hline
$100$ & 6.66667 e-9 &1.70765 e-10 &2.00970 e-12  \\
$200$ & 2.66667 e-8 &3.41530 e-10 &1.76840 e-12  \\
$300$ & 6.00000 e-8 &5.12295 e-10 &1.69300 e-12  \\
$400$ & 1.06667 e-7 &6.83060 e-10 &1.65600 e-12  \\
$500$ & 1.66667 e-7&8.53825 e-10 &1.63410 e-12  \\
$600$ & 2.40000 e-7 &1.02459 e-9 &1.61950 e-12  \\
$700$ & 3.26667 e-7  &1.19536 e-9 &1.60920 e-12  \\
$800$ & 4.26667 e-7 &1.36612 e-9 &1.60150 e-12  \\
$900$ & 5.40000 e-7 &1.53689 e-9 &1.59550 e-12  \\
$1\,000$ & 6.66667 e-7 &1.70765 e-9 &1.59070 e-12  \\

\hline
\end{tabular}
\end{table}

\section{Sources of improvements}\label{improvements}

 \begin{enumerate}
 \item {\bf About the zero-density:}
\begin{enumerate}
\item For each $\sigma$ needed, we re-optimize the zero density results of \cite[Theorem 1.1]{KLN18} in Theorem \ref{ZD-KLN} and use the new RH verification of Platt-Trudgian \cite{PlaTru21RH}. 
The effect of this improvement is small but noticeable.  
Additionally, we implement the improvement to the 
$C_8$ term of  \cite[Theorem 1.1]{KLN18} as suggested by \cite[Section 6]{CulJoh21} using \cite[Theorem 2]{CT21}. 
The effect is to divide the term $C_8$ by $2$. This does not significantly impact the results of the current work.
We provide a selection of values in Table \ref{table:NSigmaOptTable}.
 
 \item We extend the zero density result of \cite{KLN18} to one which is valid on intervals (see Corollary \ref{ZDinterval}). This result allows us to set $\sigma_2$ as a function of $x$ in the asymptotic formula.
 The careful selection of $\sigma_2$ allows us to get $C=2$ and $B=3/2$ significantly improving the asymptotic behaviour of the bound.
   We provide a selection of values in Table \ref{table:NSigmaOptTable2}.
\end{enumerate}

 \item {\bf About vertically splitting the zeros:}\\
Splitting the region vertically once at $\sigma$ tends to result in a tension between optimizing the selection of $T$ and of $\sigma$. 
By further splitting the regions at $\sigma_1$ and $\sigma_2$ (as explained in Section \ref{section-split}), we eliminate this tension. This renders negligible the contribution of zeros near the half line (see bound $\varepsilon_2$ for the sum over the zeros $\Sigma_{0}^{\sigma_1}$ in \eqref{def-epsilon2}), thereby allowing larger values for both $\sigma_2$ and $T$. This reduces the contribution of the zeros beyond $\sigma_2$ (see bound $\varepsilon_4$ for $\Sigma_{\sigma_2}^1$ in \eqref{def-epsilon4}) and thus improving significantly both numerical  and asymptotic results.  

 \item {\bf About numerically approximating Riemann-Stieltjes integrals:}
\begin{enumerate}
\item In Proposition \ref{bnd-Sigma-sigma1-sigma2} we improve the use of the zero density result as a function of $\sigma$ to reduce the contribution of the zeros between $\sigma_1$ and $\sigma_2$ further. 
We do this by using a better approximation of the Riemann-Stieltjes integral defining the sum over the zeros $\Sigma_{\sigma_1}^{\sigma_2}$, rather than simply bounding the integrand by a uniform upper bound. 
Because in this case the integral is impractical to evaluate explicitly, we use a `telescoping' strategy which yields a numerically tight upper bound.
The numerical nature of this optimization means this is important only in the numerical results rather than the asymptotic results. The effect of this improvement is small but noticeable in $\varepsilon_3$.
 
 \item We improve the use of the zero density results as a function of the height of the zero  (Lemma \ref{B4Bound} and Proposition \ref{prop:eps4bound}). By using a better approximation of the Riemann-Stieltjes integral, rather than 
 simply bounding the integrand by a uniform upper bound, we can significantly reduce the over count. The effect of this improvement is significant as it impacts $\varepsilon_4$ which is the main contribution to the error terms in Theorems \ref{thm-numericVersion} and \ref{thm-asymptoticVersion} (see Table \ref{TableResultsNumeric} and \eqref{def-Ax}.)
\end{enumerate}

\item {\bf About the zeros close to the half line:}\\
We use a symmetry argument to reduce the bound on the contribution of the zeros near, but not on, the half line (see Proposition \ref{bnd-Sigma-0-sigma1}).

 \item {\bf About reciprocal sums over zeros on the $1/2$-line:}\\
 We improve estimates on reciprocal sums over zeros. Firstly in Lemma \ref{b1Bnd} we give a bound on reciprocal sums of zeros in an interval. 
Additionally we use explicit computations with lists of zeros to obtain sharp estimates of the reciprocal sum of zeros up to $V=30\,610\,046\,000$, see Corollary \ref{prop:SOTO} and Table \ref{tab:Sig1}: 
\[39.5797 \le \sum_{ 0 < \gamma < V  } \frac{1}{\gamma}\le 39.5798.\] 
See \cite{BrPlTr21} for more refined estimates.
Note, other improvements, particularly the extra splitting of regions, see next point, render this improvement less relevant.
These improvements would be more important for smaller values of $x$ or if one is assuming RH. 

 \item {\bf Tightening the gap between numerical and asymptotic bounds:}\\
 We perform a careful analysis to show that the contribution from the zeros near the zero-free region (the $\varepsilon_4$ term) is decreasing as a function of $x$ (see Proposition \ref{epsilon4-dec-simple}).
This allows us to obtain asymptotic results similar to our numerical  ones. This gives a significant improvement in the asymptotic result of Theorem \ref{thm-asymptoticVersion}.

 \item {\bf About Perron's formula:}\\
 We note that there is an improved version of Perron's Formula (see Theorem \ref{CHJ}).
 However the changes which allow us to increase $T$ render this improvement less relevant.
For smaller $x$, simply increasing $T$ further would be more beneficial than using an alternative version of Perron's formula. In addition, other methods, such as those of B{\"u}the \cite{Bu16, Bu18}, are already significantly better for $ \log x <2\,100$.
On the other hand, as $x$ takes larger values, one can make improvements beyond the 5th digit by also using lower values of $T$ (rather than using consistently $T=  \exp\big(3 \sqrt{\frac{ \log x}R} \big)$).
 \end{enumerate}

\newpage
\section{Auxiliary Tables}
\noindent
\begin{table}[h!]
\caption{Table of values for $\varepsilon(x_0) = \varepsilon(x_0,\s_2,c, N, K) $ as defined in \eqref{def-epsilon} for Theorem \ref{thm-numericVersion}.
with values $c=3$, $\sigma_1 = 0.9$, $N=\lceil100\,000(\s_2-\s_1) \rceil$ and $K=100\,000$, $R = 5.5666305$ and using Dudek's $\varepsilon_1$.
Note: $\s_2$ is only optimized to $5$ digits.
}
\label{TableResultsNumeric}
\resizebox{\textwidth}{!}{
\begin{tabular}{|r|c|llll|l|}
\hline
$\log( x_0 )$ & $\s_2$ & $\varepsilon_1$   & $\varepsilon_2$  & $\varepsilon_3$  & $\varepsilon_4$  & $\varepsilon(x_0)$ \\
\hline
$1\,000$ & 0.99130 & 6.8931 e-12 & 2.2179 e-42 & 1.1486 e-10 & 1.2595 e-9 & 1.3812 e-9 \\
$2\,000$ & 0.99221 & 1.6115 e-18 & 2.5382 e-85 & 1.0478 e-13 & 2.3698 e-12 & 2.4746 e-12 \\
$2\,100$ & 0.99227 & 4.3625 e-19 & 1.2306 e-89 & 5.4150 e-14 & 1.2705 e-12 & 1.3246 e-12 \\
$2\,200$ & 0.99232 & 1.2152 e-19 & 5.9424 e-94 & 2.7737 e-14 & 6.8202 e-13 & 7.0976 e-13 \\
$2\,300$ & 0.99236 & 3.4763 e-20 & 2.8594 e-98 & 1.4038 e-14 & 3.6655 e-13 & 3.8059 e-13 \\
$2\,400$ & 0.99241 & 1.0198 e-20 & 1.3716 e-102 & 7.3304 e-15 & 1.9693 e-13 & 2.0426 e-13 \\
$2\,500$ & 0.99245 & 3.0626 e-21 & 6.5602 e-107 & 3.7746 e-15 & 1.0595 e-13 & 1.0972 e-13 \\
$2\,600$ & 0.99249 & 9.4049 e-22 & 3.1298 e-111 & 1.9595 e-15 & 5.7018 e-14 & 5.8978 e-14 \\
$2\,700$ & 0.99253 & 2.9495 e-22 & 1.4897 e-115 & 1.0255 e-15 & 3.0704 e-14 & 3.1729 e-14 \\
$2\,800$ & 0.99256 & 9.4362 e-23 & 7.0758 e-120 & 5.2650 e-16 & 1.6561 e-14 & 1.7087 e-14 \\
$2\,900$ & 0.99260 & 3.0766 e-23 & 3.3544 e-124 & 2.7975 e-16 & 8.9293 e-15 & 9.2091 e-15 \\
$3\,000$ & 0.99263 & 1.0213 e-23 & 1.5874 e-128 & 1.4554 e-16 & 4.8223 e-15 & 4.9678 e-15 \\
$4\,000$ & 0.99289 & 3.8012 e-28 & 8.3087 e-172 & 2.5203 e-19 & 1.0769 e-17 & 1.1021 e-17 \\
$5\,000$ & 0.99311 & 4.4810 e-32 & 3.9878 e-215 & 6.0477 e-22 & 2.9338 e-20 & 2.9942 e-20 \\
$6\,000$ & 0.99334 & 1.2102 e-35 & 1.8179 e-258 & 2.3940 e-24 & 1.2737 e-22 & 1.2976 e-22 \\
$7\,000$ & 0.99356 & 6.1586 e-39 & 8.0082 e-302 & 1.4021 e-26 & 8.3760 e-25 & 8.5162 e-25 \\
$8\,000$ & 0.99379 & 5.1936 e-42 & 3.4432 e-345 & 1.3533 e-28 & 7.6506 e-27 & 7.7860 e-27 \\
$9\,000$ & 0.99417 & 6.6323 e-45 & 1.4537 e-388 & 2.4527 e-30 & 8.9809 e-29 & 9.2262 e-29 \\
$10\,000$ & 0.99449 & 1.2006 e-47 & 6.0512 e-432 & 3.7257 e-32 & 1.3316 e-30 & 1.3688 e-30 \\
$20\,000$ & 0.99619 & 6.4252 e-70 & 6.3468 e-866 & 4.0934 e-47 & 1.8958 e-45 & 1.9367 e-45 \\
$30\,000$ & 0.99693 & 4.0605 e-87 & 4.8888 e-1300 & 1.2153 e-58 & 6.5467 e-57 & 6.6682 e-57 \\
$40\,000$ & 0.99736 & 1.1531 e-101 & 3.3291 e-1734 & 2.0196 e-68 & 1.3291 e-66 & 1.3493 e-66 \\
$50\,000$ & 0.99766 & 1.6581 e-114 & 2.1204 e-2168 & 5.6525 e-77 & 3.6804 e-75 & 3.7369 e-75 \\
$60\,000$ & 0.99787 & 3.9127 e-126 & 1.2951 e-2602 & 8.2972 e-85 & 6.5977 e-83 & 6.6806 e-83 \\
$70\,000$ & 0.99804 & 7.7353 e-137 & 7.6841 e-3037 & 6.2358 e-92 & 4.8619 e-90 & 4.9243 e-90 \\
$80\,000$ & 0.99817 & 8.2566 e-147 & 4.4645 e-3471 & 1.2079 e-98 & 1.1046 e-96 & 1.1166 e-96 \\
$90\,000$ & 0.99828 & 3.5041 e-156 & 2.5526 e-3905 & 6.4784 e-105 & 6.2867 e-103 & 6.3515 e-103 \\
$100\,000$ & 0.99838 & 4.7299 e-165 & 1.4411 e-4339 & 9.8527 e-111 & 7.7127 e-109 & 7.8112 e-109 \\
$200\,000$ & 0.99887 & 8.7978 e-237 & 3.2889 e-8682 & 1.0317 e-158 & 1.2350 e-156 & 1.2453 e-156 \\
$300\,000$ & 0.99908 & 6.2208 e-292 & 5.6126 e-13025 & 1.0986 e-195 & 2.1996 e-193 & 2.2106 e-193 \\
$400\,000$ & 0.99921 & 1.7897 e-338 & 8.5065 e-17368 & 1.5373 e-226 & 2.1209 e-224 & 2.1363 e-224 \\
$500\,000$ & 0.99929 & 1.6709 e-379 & 1.2083 e-21710 & 3.2223 e-254 & 9.6746 e-252 & 9.7068 e-252 \\
$600\,000$ & 0.99935 & 1.2951 e-416 & 1.6472 e-26053 & 3.4804 e-279 & 1.7998 e-276 & 1.8032 e-276 \\
$700\,000$ & 0.99940 & 9.4139 e-451 & 2.1829 e-30396 & 8.0982 e-302 & 3.1872 e-299 & 3.1953 e-299 \\
$800\,000$ & 0.99944 & 1.5480 e-482 & 2.8336 e-34739 & 7.0513 e-323 & 2.0918 e-320 & 2.0988 e-320 \\
$900\,000$ & 0.99947 & 2.1427 e-512 & 3.6206 e-39082 & 5.1196 e-343 & 2.6418 e-340 & 2.6470 e-340 \\
$1\,000\,000$ & 0.99950 & 1.2150 e-540 & 4.5688 e-43425 & 1.9527 e-361 & 3.9371 e-359 & 3.9566 e-359 \\

\hline 
\end{tabular}
}
\end{table}
\newpage
\begin{table}[h!]
\caption{Numerical bounds for $E_{\psi}(x)$ from Theorem \ref{thm-asymptoticVersion}: values for $A$ and $\varepsilon_{\text{asm}}$ are calculated 
with values $c = 4$, $\sigma_1 = 0.9$, $c_1=17.4194$, $K=100\,000$ and $R  = 5.5666305$. 
The values $\varepsilon(x_0)$ are from \eqref{def-epsilon} with $\s_2$ optimized with up to $7$ digits.
The final column gives an approximate lower bound for $A(x_0)$, see Lemma \ref{lem:fksinterp1}, with $t_0$ as in \eqref{def-t0-T-sigma2}. 
}\label{TableResultsAsympt}
\begin{tabular}{|r|r|l|l|r|}
\hline
$\log(x_0)$ & $A(x_0)$ &\ \ \  $\varepsilon_{\text{asm}}(x_0,x_0)$ & \ \qquad\  $\varepsilon(x_0)$ & $ \varepsilon(x_0)  \tfrac{t_0(x_0)^2}{(\log t_0(x_0))^3}$ \\
\hline
$1\,000$ & 338.3058 & 1.8586 e-6 & 1.3812 e-9 & 0.26 \\
$2\,000$ & 263.2129 & 6.1596 e-11 & 2.4746 e-12 & 10.58 \\
$3\,000$ & 233.0775 & 1.9985 e-14 & 4.9678 e-15 & 57.94 \\
$4\,000$ & 215.8229 & 2.1643 e-17 & 1.1021 e-17 & 109.90 \\
$5\,000$ & 204.2929 & 5.1120 e-20 & 2.9942 e-20 & 119.66 \\
$6\,000$ & 195.8842 & 2.1111 e-22 & 1.2976 e-22 & 120.41 \\
$7\,000$ & 189.3959 & 1.3350 e-24 & 8.5162 e-25 & 120.83 \\
$8\,000$ & 184.1882 & 1.1849 e-26 & 7.7860 e-27 & 121.03 \\
$9\,000$ & 179.8849 & 1.3891 e-28 & 9.2262 e-29 & 119.48 \\
$10\,000$ & 176.2484 & 2.0574 e-30 & 1.3688 e-30 & 117.27 \\
$20\,000$ & 156.4775 & 2.9117 e-45 & 1.9367 e-45 & 104.09 \\
$30\,000$ & 147.5424 & 1.0041 e-56 & 6.6682 e-57 & 97.99 \\
$40\,000$ & 142.1006 & 2.0344 e-66 & 1.3493 e-66 & 94.25 \\
$50\,000$ & 138.3136 & 5.6411 e-75 & 3.7369 e-75 & 91.63 \\
$60\,000$ & 135.4686 & 1.0095 e-82 & 6.6806 e-83 & 89.66 \\
$70\,000$ & 133.2221 & 7.4474 e-90 & 4.9243 e-90 & 88.09 \\
$80\,000$ & 131.3849 & 1.6899 e-96 & 1.1166 e-96 & 86.82 \\
$90\,000$ & 129.8428 & 9.6183 e-103 & 6.3515 e-103 & 85.75 \\
$100\,000$ & 128.5221 & 1.1835 e-108 & 7.8112 e-109 & 84.84 \\
$200\,000$ & 121.0360 & 1.8921 e-156 & 1.2453 e-156 & 79.67 \\
$300\,000$ & 117.4647 & 3.3594 e-193 & 2.2106 e-193 & 77.30 \\
$400\,000$ & 115.2251 & 3.2459 e-224 & 2.1363 e-224 & 75.84 \\
$500\,000$ & 113.6357 & 1.4732 e-251 & 9.7068 e-252 & 74.88 \\
$600\,000$ & 112.4241 & 2.7311 e-276 & 1.8032 e-276 & 74.23 \\
$700\,000$ & 111.4565 & 4.8387 e-299 & 3.1953 e-299 & 73.61 \\
$800\,000$ & 110.6577 & 3.1764 e-320 & 2.0988 e-320 & 73.12 \\
$900\,000$ & 109.9819 & 3.9988 e-340 & 2.6470 e-340 & 72.81 \\
$10^{6}$ & 109.3992 & 5.9745 e-359 & 3.9566 e-359 & 72.45 \\
$10^{7}$ & 100.5097 & 1.6190 e-1153 & 1.0308 e-1153 & 64.00 \\
$10^{8}$ & 96.0345 & 2.6351 e-3669 & 1.6721 e-3669 & 60.94 \\
$10^{9}$ & 93.6772 & 4.0506 e-11628 & 2.6089 e-11628 & 60.34 \\
\hline 
\end{tabular}
\end{table}
\newpage
%
\begin{table}[h!]
\centering
\caption{Sample of values of $C_1$, $C_2$ as defined in Theorem \ref{ZD-KLN} for $H_0 = 3 \cdot 10^{12}$, $k = 1$, $\mu = 1.2362...$, and with optimized parameters
$d, \delta,$ and $\alpha$.\\This table is computed using the subconvexity bound $a_1=0.77$, see Remark \ref{compprev}(\ref{subconv}).}\label{table:NSigmaOptTable}.
\begin{tabular}{|c|c|c|c|c|c|c|c|}
\hline
$\sigma$ & $\alpha$ & $\delta$ & $d$ & $\mathcal{C}_1$ & $c_1$ & $\mathcal{C}_2$ & $c_2$ \\
\hline
0.600 & 0.2745 & 0.3089 & 0.3411 & 6.2364 & 2.9101 & 11.3289 & 5.2863 \\
0.700 & 0.2166 & 0.3088 & 0.3398 & 10.3372 & 4.8414 & 9.5857 & 4.4894 \\
0.800 & 0.1594 & 0.3089 & 0.3382 & 16.5703 & 7.7978 & 7.8423 & 3.6905 \\
0.900 & 0.1041 & 0.3092 & 0.3360 & 25.3876 & 12.0272 & 6.0986 & 2.8891 \\
0.910 & 0.0988 & 0.3093 & 0.3357 & 26.4101 & 12.5219 & 5.9242 & 2.8088 \\
0.920 & 0.0935 & 0.3094 & 0.3354 & 27.4551 & 13.0286 & 5.7497 & 2.7285 \\
0.930 & 0.0883 & 0.3094 & 0.3351 & 28.5212 & 13.5467 & 5.5753 & 2.6481 \\
0.940 & 0.0831 & 0.3095 & 0.3348 & 29.6068 & 14.0756 & 5.4009 & 2.5677 \\
0.950 & 0.0780 & 0.3096 & 0.3344 & 30.7097 & 14.6145 & 5.2264 & 2.4872 \\
0.960 & 0.0730 & 0.3098 & 0.3341 & 31.8278 & 15.1623 & 5.0520 & 2.4067 \\
0.970 & 0.0681 & 0.3099 & 0.3337 & 32.9584 & 15.7180 & 4.8775 & 2.3261 \\
0.980 & 0.0633 & 0.3101 & 0.3333 & 34.0986 & 16.2804 & 4.7031 & 2.2455 \\
0.990 & 0.0586 & 0.3103 & 0.3329 & 35.2449 & 16.8481 & 4.5286 & 2.1648 \\
0.991 & 0.0582 & 0.3103 & 0.3329 & 35.3598 & 16.9051 & 4.5111 & 2.1567 \\
0.992 & 0.0577 & 0.3103 & 0.3329 & 35.4746 & 16.9621 & 4.4937 & 2.1486 \\
0.993 & 0.0572 & 0.3103 & 0.3328 & 35.5895 & 17.0192 & 4.4762 & 2.1406 \\
0.994 & 0.0568 & 0.3103 & 0.3328 & 35.7044 & 17.0762 & 4.4588 & 2.1325 \\
0.995 & 0.0563 & 0.3104 & 0.3327 & 35.8193 & 17.1334 & 4.4413 & 2.1244 \\
0.996 & 0.0559 & 0.3104 & 0.3327 & 35.9342 & 17.1905 & 4.4239 & 2.1163 \\
0.997 & 0.0554 & 0.3104 & 0.3326 & 36.0492 & 17.2477 & 4.4064 & 2.1083 \\
0.998 & 0.0550 & 0.3104 & 0.3326 & 36.1641 & 17.3049 & 4.3890 & 2.1002 \\
0.999 & 0.0545 & 0.3104 & 0.3326 & 36.2790 & 17.3621 & 4.3715 & 2.0921 \\
\hline
\end{tabular}
\vskip6pt
\caption*{The table below is computed using the subconvexity bound $a_1=0.618$, see Remark \ref{compprev}(\ref{subconv})}
\begin{tabular}{|c|c|c|c|c|c|c|c|}
\hline
$\sigma$ & $\alpha$ & $\delta$ & $d$ & $\mathcal{C}_1$ & $c_1$ & $\mathcal{C}_2$ & $c_2$ \\
\hline
0.600 & 0.2746 & 0.3117 & 0.3447 & 4.3794 & 2.0222 & 11.3310 & 5.2321 \\
0.700 & 0.2167 & 0.3116 & 0.3434 & 7.9114 & 3.6667 & 9.5879 & 4.4437 \\
0.800 & 0.1595 & 0.3117 & 0.3417 & 13.8214 & 6.4369 & 7.8445 & 3.6533 \\
0.900 & 0.1042 & 0.3120 & 0.3394 & 23.0787 & 10.8209 & 6.1007 & 2.8604 \\
0.910 & 0.0988 & 0.3121 & 0.3392 & 24.2157 & 11.3634 & 5.9263 & 2.7810 \\
0.920 & 0.0935 & 0.3122 & 0.3389 & 25.3914 & 11.9255 & 5.7519 & 2.7015 \\
0.930 & 0.0883 & 0.3123 & 0.3386 & 26.6053 & 12.5070 & 5.5774 & 2.6219 \\
0.940 & 0.0831 & 0.3124 & 0.3382 & 27.8565 & 13.1077 & 5.4030 & 2.5423 \\
0.950 & 0.0781 & 0.3125 & 0.3379 & 29.1439 & 13.7272 & 5.2285 & 2.4627 \\
0.960 & 0.0731 & 0.3126 & 0.3375 & 30.4659 & 14.3650 & 5.0541 & 2.3831 \\
0.970 & 0.0682 & 0.3127 & 0.3372 & 31.8207 & 15.0203 & 4.8796 & 2.3033 \\
0.980 & 0.0634 & 0.3129 & 0.3368 & 33.2058 & 15.6923 & 4.7051 & 2.2235 \\
0.990 & 0.0587 & 0.3131 & 0.3364 & 34.6186 & 16.3799 & 4.5307 & 2.1437 \\
0.991 & 0.0582 & 0.3131 & 0.3363 & 34.7613 & 16.4495 & 4.5132 & 2.1357 \\
0.992 & 0.0577 & 0.3131 & 0.3363 & 34.9042 & 16.5192 & 4.4958 & 2.1277 \\
0.993 & 0.0573 & 0.3131 & 0.3362 & 35.0474 & 16.5890 & 4.4783 & 2.1197 \\
0.994 & 0.0568 & 0.3132 & 0.3362 & 35.1908 & 16.6590 & 4.4609 & 2.1117 \\
0.995 & 0.0564 & 0.3132 & 0.3362 & 35.3344 & 16.7292 & 4.4434 & 2.1037 \\
0.996 & 0.0559 & 0.3132 & 0.3361 & 35.4782 & 16.7994 & 4.4260 & 2.0958 \\
0.997 & 0.0555 & 0.3132 & 0.3361 & 35.6223 & 16.8698 & 4.4085 & 2.0878 \\
0.998 & 0.0550 & 0.3132 & 0.3360 & 35.7666 & 16.9404 & 4.3911 & 2.0798 \\
0.999 & 0.0546 & 0.3133 & 0.3360 & 35.9111 & 17.0110 & 4.3736 & 2.0718 \\
\hline
\end{tabular}
\end{table}
\newpage
\begin{table}[h!]
\centering
\caption{Sample of values for $\tilde{c_1},\tilde{c_2}$ as in Corollary \ref{ZDinterval} for $H_0 = 3 \cdot 10^{12}$ and $k = 1$ with $\mu = 1.2362\ldots$ and $C_7(\eta, H_0) = 17.424$ and optimized parameters
$\alpha, \delta$, and $d$.\\ This table is computed using the subconvexity bound $a_1=0.77$, see Remark \ref{compprev}(\ref{subconv}).}\label{table:NSigmaOptTable2}
\begin{tabular}{cccccccccc}
\hline
$\s_1$ & $\s_2$ & $\alpha$ & $\delta$ & $d$ & $\mathcal{C}_1$ & $\tilde{c_1}$ & $\mathcal{C}_2$ & $\tilde{c_2}$ \\
\hline
0.60 & 0.65 & 0.2456 & 0.3089 & 0.3405 & 8.0587 & 3.7669 & 11.3285 & 5.2954 \\
0.65 & 0.70 & 0.2167 & 0.3089 & 0.3399 & 10.3373 & 4.8415 & 10.4569 & 4.8975 \\
0.70 & 0.75 & 0.1879 & 0.3089 & 0.3391 & 13.1505 & 6.1727 & 9.5853 & 4.4992 \\
0.75 & 0.80 & 0.1595 & 0.3089 & 0.3383 & 16.5704 & 7.7979 & 8.7136 & 4.1006 \\
0.80 & 0.81 & 0.1538 & 0.3089 & 0.3381 & 17.3322 & 8.1610 & 7.8423 & 3.6926 \\
0.81 & 0.82 & 0.1482 & 0.3090 & 0.3379 & 18.1208 & 8.5373 & 7.6679 & 3.6126 \\
0.82 & 0.83 & 0.1426 & 0.3090 & 0.3377 & 18.9362 & 8.9269 & 7.4935 & 3.5326 \\
0.83 & 0.84 & 0.1370 & 0.3090 & 0.3374 & 19.7785 & 9.3298 & 7.3192 & 3.4526 \\
0.84 & 0.85 & 0.1314 & 0.3090 & 0.3372 & 20.6478 & 9.7461 & 7.1448 & 3.3725 \\
0.85 & 0.86 & 0.1259 & 0.3091 & 0.3370 & 21.5438 & 10.1759 & 6.9704 & 3.2924 \\
0.86 & 0.87 & 0.1204 & 0.3091 & 0.3368 & 22.4663 & 10.6191 & 6.7960 & 3.2123 \\
0.87 & 0.88 & 0.1150 & 0.3092 & 0.3365 & 23.4149 & 11.0755 & 6.6216 & 3.1321 \\
0.88 & 0.89 & 0.1095 & 0.3092 & 0.3363 & 24.3889 & 11.5450 & 6.4473 & 3.0519 \\
0.89 & 0.90 & 0.1041 & 0.3093 & 0.3360 & 25.3877 & 12.0272 & 6.2729 & 2.9717 \\
0.90 & 0.91 & 0.09880 & 0.3093 & 0.3357 & 26.4101 & 12.5220 & 6.0984 & 2.8915 \\
0.91 & 0.92 & 0.09350 & 0.3094 & 0.3354 & 27.4552 & 13.0287 & 5.9240 & 2.8112 \\
0.92 & 0.93 & 0.08830 & 0.3095 & 0.3351 & 28.5213 & 13.5468 & 5.7496 & 2.7309 \\
0.93 & 0.94 & 0.08310 & 0.3096 & 0.3348 & 29.6068 & 14.0757 & 5.5752 & 2.6506 \\
0.94 & 0.95 & 0.07810 & 0.3097 & 0.3345 & 30.7098 & 14.6145 & 5.4007 & 2.5702 \\
0.95 & 0.96 & 0.07310 & 0.3098 & 0.3341 & 31.8279 & 15.1623 & 5.2263 & 2.4897 \\
0.96 & 0.97 & 0.06820 & 0.3100 & 0.3338 & 32.9585 & 15.7181 & 5.0518 & 2.4093 \\
0.97 & 0.98 & 0.06340 & 0.3101 & 0.3334 & 34.0986 & 16.2805 & 4.8774 & 2.3287 \\
0.98 & 0.99 & 0.05870 & 0.3103 & 0.3330 & 35.2450 & 16.8481 & 4.7029 & 2.2481 \\
0.99 & 1.0 & 0.05410 & 0.3105 & 0.3326 & 36.3939 & 17.4194 & 4.5284 & 2.1675 \\
\hline
\end{tabular}

\vskip6pt
\caption*{The table below is computed using the subconvexity bound $a_1=0.618$, see Remark \ref{compprev}(\ref{subconv})}
\begin{tabular}{cccccccccc}
\hline
$\s_1$ & $\s_2$ & $\alpha$ & $\delta$ & $d$  & $\mathcal{C}_1$ & $\tilde{c_1}$ & $\mathcal{C}_2$ & $\tilde{c_2}$ \\
0.60 & 0.70 & 0.2167 & 0.3117 & 0.3434 & 7.9115 & 3.6668 & 11.3303 & 5.2513 \\
0.70 & 0.80 & 0.1595 & 0.3117 & 0.3418 & 13.8214 & 6.4369 & 9.5869 & 4.4648 \\
0.80 & 0.81 & 0.1539 & 0.3118 & 0.3416 & 14.5818 & 6.7949 & 7.8444 & 3.6554 \\
0.81 & 0.82 & 0.1483 & 0.3118 & 0.3414 & 15.3770 & 7.1697 & 7.6700 & 3.5762 \\
0.82 & 0.83 & 0.1427 & 0.3118 & 0.3412 & 16.2078 & 7.5617 & 7.4957 & 3.4971 \\
0.83 & 0.84 & 0.1371 & 0.3118 & 0.3410 & 17.0751 & 7.9713 & 7.3213 & 3.4179 \\
0.84 & 0.85 & 0.1315 & 0.3119 & 0.3407 & 17.9796 & 8.3991 & 7.1469 & 3.3387 \\
0.85 & 0.86 & 0.1260 & 0.3119 & 0.3405 & 18.9219 & 8.8453 & 6.9725 & 3.2594 \\
0.86 & 0.87 & 0.1205 & 0.3119 & 0.3403 & 19.9027 & 9.3103 & 6.7982 & 3.1801 \\
0.87 & 0.88 & 0.1150 & 0.3120 & 0.3400 & 20.9223 & 9.7945 & 6.6238 & 3.1008 \\
0.88 & 0.89 & 0.1096 & 0.3120 & 0.3398 & 21.9809 & 10.2980 & 6.4494 & 3.0215 \\
0.89 & 0.90 & 0.1042 & 0.3121 & 0.3395 & 23.0788 & 10.8209 & 6.2750 & 2.9422 \\
0.90 & 0.91 & 0.09890 & 0.3121 & 0.3392 & 24.2157 & 11.3635 & 6.1006 & 2.8628 \\
0.91 & 0.92 & 0.09360 & 0.3122 & 0.3389 & 25.3914 & 11.9256 & 5.9261 & 2.7833 \\
0.92 & 0.93 & 0.08840 & 0.3123 & 0.3386 & 26.6053 & 12.5071 & 5.7517 & 2.7039 \\
0.93 & 0.94 & 0.08320 & 0.3124 & 0.3383 & 27.8566 & 13.1078 & 5.5773 & 2.6244 \\
0.94 & 0.95 & 0.07810 & 0.3125 & 0.3379 & 29.1440 & 13.7273 & 5.4028 & 2.5448 \\
0.95 & 0.96 & 0.07310 & 0.3126 & 0.3376 & 30.4660 & 14.3650 & 5.2284 & 2.4653 \\
0.96 & 0.97 & 0.06820 & 0.3128 & 0.3372 & 31.8207 & 15.0203 & 5.0539 & 2.3856 \\
0.97 & 0.98 & 0.06340 & 0.3129 & 0.3368 & 33.2059 & 15.6924 & 4.8794 & 2.3059 \\
0.98 & 0.99 & 0.05870 & 0.3131 & 0.3364 & 34.6187 & 16.3800 & 4.7049 & 2.2262 \\
0.99 & 1.0 & 0.05420 & 0.3133 & 0.3360 & 36.0559 & 17.0819 & 4.5304 & 2.1464 \\
\hline
\hline
\end{tabular}
\end{table}
\end{document}